\documentclass[12pt]{amsart}
\usepackage{amsmath,amssymb,amsfonts,amsthm,amscd,amstext,amsxtra,amsopn,array,url,verbatim,mathrsfs,enumerate,anysize,soul}
\usepackage{graphicx}
\usepackage{amsmath,amssymb,amsfonts,amsthm,amssymb,amscd,url,amstext,amsxtra,amsopn
,txfonts}
\usepackage{verbatim}
\marginsize{2.25cm}{2.25cm}{2.25cm}{2.25cm}
\usepackage[dvipsnames,usenames]{xcolor}
\usepackage[colorlinks=true,urlcolor=Black,citecolor=Black,linkcolor=Black]{hyperref}
\usepackage{times}
\usepackage{mathabx}

\usepackage{enumitem} 

\makeatletter
\@namedef{subjclassname@2010}{%
  \textup{2010} Mathematics Subject Classification}
\makeatother

\newtheorem*{rep@theorem}{\rep@title}
\newcommand{\newreptheorem}[2]{%
\newenvironment{rep#1}[1]{%
 \def\rep@title{#2 \ref{##1}}%
 \begin{rep@theorem}}%
 {\end{rep@theorem}}}
\makeatother

\usepackage{amsrefs}
\newenvironment{rezabib}
  {\bibdiv\biblist\setupbib}
  {\endbiblist\endbibdiv}

\def\setupbib{\catcode`@=\active}
\begingroup\lccode`~=`@
  \lowercase{\endgroup\def~}#1#{\gatherkey{#1}}
\def\gatherkey#1#2{\gatherkeyaux{#1}#2\gatherkeyaux}
\def\gatherkeyaux#1#2,#3\gatherkeyaux{\bib{#2}{#1}{#3}}

\newtheorem{thm}{Theorem}[section]
\newtheorem{thrm}[thm]{Theorem}
\newtheorem{prop}[thm]{Proposition}
\newtheorem{lem}[thm]{Lemma}
\newtheorem{cor}[thm]{Corollary}

\theoremstyle{definition}

\newtheorem{remarks}[thm]{Remarks}

\newtheorem{remark}[thm]{Remark}

\DeclareMathOperator{\ld}{\rm Ld}

\makeatletter
\def\imod#1{\allowbreak\mkern7mu({\operator@font mod}\,\,#1)}
\makeatother

\numberwithin{equation}{section}

\begin{document}

\title[Value-Distribution of Logarithmic Derivatives of Quadratic Twists]{Value-Distribution of Logarithmic Derivatives of Quadratic Twists of Automorphic  $L$-functions}

\thanks{Research of the both authors is partially supported by NSERC}

\date{\today}



\author{Amir Akbary}
\author{Alia Hamieh}

\address{Department of Mathematics and Computer Science \\
        University of Lethbridge \\
        Lethbridge, AB T1K 3M4 \\
        Canada}
        
        \address{Department of Mathematics and Statistics \\
        University of Northern British Columbia \\
        Prince George, BC V2N4Z9 \\
        Canada}

\email{amir.akbary@uleth.ca}
\email{alia.hamieh@unbc.ca}

\begin{abstract}

Let $d\in\mathbb{N}$, and let $\pi$ be a fixed cuspidal automorphic representation of $\mathrm{GL}_{d}(\mathbb{A}_{\mathbb{Q}})$ with unitary central character. We determine the limiting distribution of the family of values $-\frac{L'}{L}(1+it,\pi\otimes\chi_D)$ as $D$ varies over fundamental discriminants. 
  Here, $t$ is a fixed real number and $\chi_D$ is the real character associated with $D$. We establish an upper bound on the discrepancy in the convergence of this family to its limiting distribution. As an application of this result, we obtain an upper bound on the small values of $\left|\frac{L'}{L}(1,\pi\otimes\chi_D)\right|$ when $\pi$ is self-dual.

\end{abstract}

\maketitle

\section{Introduction}

Bohr and Jessen  showed in \cite{BJ} that the values $\log{\zeta(\sigma+it)}$  for a fixed $\sigma>\frac12$ as $t$ varies in $\mathbb{R}$ have a limiting distribution with a continuous density in the complex plane. In \cite{JW}, Jessen and Wintner revisited this problem from a more general perspective using  ideas from probability theory and  Fourier analysis machinery, which  allowed them to reveal detailed information on the distribution function in Bohr-Jessen's theorem (see \cite[Theorem 19]{JW}). In recent years, this line of research was pursued further by many authors studying the distribution of $\log\zeta(\sigma+it)$ in the critical strip (e.g. \cite{Harman-Mat}, \cite{HM}, \cite{lamzouri2} and \cite{Lf}). We focus here on the work of Lamzouri, Lester, and Radziwi\l \l{ }\cite{Lf} in which the authors investigate the discrepancy between the distributions of $\log\zeta(\sigma+it)$ and that of an adequately chosen random series $\log\zeta(\sigma,\mathbb{X})$. More precisely, let $\{\mathbb{X}_p\}_{p\;\text{prime}}$ be a sequence of independent random variables uniformly distributed on the unit circle. Consider the random Euler product $\zeta(\sigma,\mathbb{X})=\prod_{p}\left(1-\mathbb{X}_{p}p^{-\sigma}\right)^{-1}$ which converges almost surely for $\sigma>\frac12$. We can ask whether $\zeta(\sigma,\mathbb{X})$ is a good model for the Riemann zeta function. The authors of \cite{Lf} answer this question affirmatively.

 \begin{thm}\cite[Theorem~1.1]{Lf} Let $\frac12<\sigma<1$ be fixed. Then we have \begin{align*}&\mathbb{D}_{\sigma}(\log{\zeta}; T):=\sup_{\mathcal{R}\subset \mathbb{C}}\left|\mathbb{P}_{T}\left(\log\zeta(\sigma+it)\in \mathcal{R}\right)-\mathbb{P}\left(\log\zeta(\sigma,\mathbb{X})\in \mathcal{R}\right)\right|\ll \frac{1}{(\log T)^{\sigma}}.
\end{align*}
For $\sigma=1$, we have 
\begin{align*}&\mathbb{D}_{1}(\log{\zeta}; T):=\sup_{\mathcal{R}\subset \mathbb{C}}\left|\mathbb{P}_{T}\left(\log\zeta(1+it)\in \mathcal{R}\right)-\mathbb{P}\left(\log\zeta(1,\mathbb{X})\in \mathcal{R}\right)\right|\ll \frac{\log\log T}{\log T}.
\end{align*}
Here $\mathcal{R}$ varies over all rectangles in $\mathbb{C}$ with sides parallel to the axes, and \[\mathbb{P}_{T}\left(f(t)\in \mathcal{R}\right)=\frac{1}{T}\mathrm{meas}\{T\leq t\leq2T:f(t)\in \mathcal{R}\}.\]
\end{thm}
This improves on the result in \cite{Harman-Mat} where the authors prove that for any $\epsilon>0$, we have
 \[\mathbb{D}_{\sigma}(\log{\zeta}; T)\ll \frac{1}{(\log T)^{\frac{4\sigma-2}{21+8\sigma}-\epsilon}},\] provided that $\frac12<\sigma\leq 1$.
Recently Xiao and Zhai extended the result of \cite{Lf}, for $\frac12<\sigma<1$ to the case of $L$-functions attached to Hecke eigenforms of level 1.  

In another direction, Mine \cite{M20} studied the discrepancy $D_s(\log{L_f}; q)$ in the value distribution at a fixed point $s=\sigma+it$ of averages and harmonic averages of the values of logarithm of the $L$-functions $L_f(s)$ attached to the primitive cusp forms $f$ of weight 2 and prime level $q$, as $q\rightarrow \infty$. In \cite[Theorem~1.4]{M20} it is proved that

$$\mathbb{D}_s(\log{L_f}; q)\ll 
\begin{cases}
(\log\log{q})/\log{q}&{if~\sigma>1,}\\
(\log\log{q}) (\log\log\log{q})/\log{q}&{if~\sigma=1,}\\
1/(\log{q})^\sigma&{if~\frac12<\sigma<1,}
\end{cases}$$
generalizing a special case of a 1-dimensional result of Cogdell and Michel \cite[Corollary~1.16]{CM04} and a 1-dimensional discrepancy estimate of Golubeva \cite[Theorem 1]{G06}. The reader is referred to \cite[Equations~(1.16) and (1.17)]{M20} for the exact definition of  
$\mathbb{D}_s(\log{L_f}; q)$.

Inspired by studying the small values of the Euler-Kronecker constants of the cyclotomic fields $\mathbb{Q}(\zeta_q)$, Lamzouri and Languasco \cite{LL21} proved a discrepancy estimate for the distribution of $\frac{L^\prime}{L}(1, \chi)$ where $\chi$ varies over non-trivial Dirichlet characters mod $q$ (prime). By defining a suitable random series $\ld(1, \mathbb{X})$ associated to a certain random sequence 
$\mathbb{X}$ and setting
$$\mathbb{D}_1\left({L_{\chi}^{\prime}}/{L_{\chi}}; q\right)= \sup_{\mathcal{R}\subset{\mathbb{C}}} \left| \frac{1}{q-1} \left|\left\{\chi\neq \chi_0~{\rm mod}~q:~\frac{L^\prime}{L}(1, \chi)\in \mathcal{R}   \right\}   \right|-\mathbb{P} (\ld(1, \mathbb{X})\in \mathcal{R}  \right|,$$
where the supremum is taken over all rectangles of the 
complex plane with sides parallel to the coordinate axes, they proved in \cite[Theorem 1.5]{LL21} that
$$\mathbb{D}_1\left({L_{\chi}^{\prime}}/{L_{\chi}}; q\right) \ll \frac{(\log\log{q})^2}{\log{q}},\quad \text{as}\; q\to\infty;~ q\;\text{prime}.$$

In \cite{Hamieh-Mcclenagan}, Hamieh and Mcclenagan,  determined an asymptotic distribution function for the values $\frac{L'}{L}(\sigma,\chi_D)$ as $D$ varies over all fundamental discriminants $D$, with $|D|\leq N$, for a fixed real number $\frac12<\sigma<1$, removing the dependence on GRH in a result of Mourtada and Murty \cite{M-M}. 
Here $\chi_D=\left(\frac{D}{\cdot}\right)$ is the Kronecker symbol for $D$, and $L(s,\chi_D)$ is the associated Dirichlet $L$-function. 
In addition, in \cite[Theorem 1.3]{Hamieh-Mcclenagan} they proved
$$\mathbb{D}_\sigma\left({L_{\chi_D}^{\prime}}/{L_{\chi_D}}; N\right)\ll \left(\frac{\log\log{N}}{\log{N}}  \right)^\sigma, $$
where $$\mathbb{D}_\sigma \left({L_{\chi_D}^{\prime}}/{L_{\chi_D}}; N\right)=\sup_{z\in \mathbb{R}} \left|\mathbb{P}_N\left(\frac{L^{\prime}}{L}(\sigma, \chi_D) \leq z\right)- \mathbb{P}\left(\ld (\sigma, \mathbb{X}) \leq z\right)
\right|$$
is the discrepancy between the value-distribution of quadratic twists and that of a random series $\ld (\sigma, \mathbb{X})$ attached to a random sequence $\mathbb{X}$ described in the introduction of \cite{Hamieh-Mcclenagan}.
Here
 \begin{align*}&\mathbb{P}_{N}\left(\frac{L'}{L}(\sigma,\chi_{D})\leq z\right)
 =\frac{1}{\left|\mathcal{F}(N)\right|} \left|  \left\{D \in \mathcal{F}(N):~\frac{L'}{L}(\sigma, \chi_D)\leq z \right \}\right|,\end{align*}
where $\mathcal{F}(N)$ denotes the collection of the fundamental discriminants $D$ with $|D|\leq N$. Note that we have (see \cite[page~1017]{GS}) \[|\mathcal{F}(N)|=\frac{6}{\pi^2}N+O\left(N^{\frac12}\right).\]

In this paper, we consider the analogous problem for logarithmic derivatives of quadratic twists of automorphic $L$-functions. Let $\pi$ be a cuspidal automorphic representation of $\mathrm{GL}_{d}(\mathbb{A}_{\mathbb{Q}})$ with unitary central character. Let $L(s, \pi)=\sum_{n=1}^{\infty}a_{\pi}(n)n^{-s}$ be the associated Dirichlet series. In particular, we have 
\begin{equation}
L(s, \pi)=\prod_{p}\prod_{j=1}^d\left(1-\frac{\alpha_{j, \pi}(p)}{p^s}\right)^{-1},\end{equation}
where $\alpha_{j, \pi} (p)$'s are the \emph{Satake parameters} of $\pi$.
Thus, $a_{\pi}(p)={\sum_{j=1}^d\alpha_{j, \pi}(p)}$. By a result of Rudnick and Sarnak \cite{R-S}, we know that 
\begin{equation}
\label{RS-bound}
|\alpha_{j,\pi}(p)|\leq p^{\frac12-\frac{1}{d^2+1}}
\end{equation}
for all $j\in\{1,2,\dots,d\}$. We set 
\begin{equation}
\lambda_{\pi}(n)=\begin{cases}\displaystyle{\sum_{j=1}^d{\alpha_{j, \pi}(p)}^m}&\;\text{if}\;
n=p^m,\\
0&\;\text{otherwise.}\end{cases}
\end{equation}

Corresponding to a representation $\pi$, there is a dual representation $\tilde{\pi}$. The collection of the Satake parameters for $\tilde{\pi}$ coincides with the collection of the complex conjugates of the Satake parameters for $\pi$, and thus $a_{\tilde{\pi}}(n)= \overline{a_{\pi}(n)}$. We call a cuspidal representation $\pi$ \emph{self-dual} if ${\pi}\simeq \tilde{\pi}$.

For $\Re(s)>1$ and a fundamental discriminant $D\in\mathcal{F}(N)$, we set
\begin{equation}
L(s, \pi\otimes \chi_D)=\sum_{n=1}^{\infty}\frac{a_\pi(n)\chi_D(n)}{n^s}
\end{equation}
and
\begin{equation}
-\frac{L^\prime}{L}(s, \pi\otimes \chi_D)=\sum_{n=1}^{\infty} \frac{\Lambda(n)\lambda_\pi(n) \chi_D(n)}{n^s}.
\end{equation}
These functions have meromorphic continuations to the entire complex plane.

Let $\mathcal{R}$ be a rectangle in $\mathbb{C}$ with sides parallel to the coordinate axes. We denote by $\mathbf{1}_{\mathcal{R}}\left(\cdot\right)$ the characteristic function of $\mathcal{R}$. For $t\in \mathbb{R}$ we define
 \begin{align*}&\mathbb{P}_{N}\left(-\frac{L'}{L}(1+it,\pi\otimes\chi_{D})\in \mathcal{R}\right)
 =\frac{1}{\left|\mathcal{F}(N)\right|}\sum_{D\in\mathcal{F}(N)}
\mathbf{1}_{\mathcal{R}}\left(-\frac{L'}{L}(1+it,\pi\otimes\chi_{D})\right)
.\end{align*}
Thus,
$\mathbb{P}_{N}\left(-\tfrac{L'}{L}(1+it,\pi\otimes\chi_{D})\in \mathcal{R}\right)$ is the proportion of the fundamental discriminants $D$ for which $-\tfrac{L'}{L}(1+it,\pi\otimes\chi_{D})\in \mathcal{R}$.


Let us now introduce the probabilistic random model which we use to approximate the distribution of the arithmetic values $-\frac{L'}{L}(1+it,\pi\otimes\chi_{D})$ described above. Consider the sequence of independent random variables $\{\mathbb{X}_p\}_{p\;\text{prime}}$ given by
	\begin{equation*}\label{eqn:random-model}\mathbb{P} \big( \mathbb{X}_p = a \big) = \begin{cases}
        \frac{p}{2(p+1)} & \text{if $a= \pm 1$}, \\
        \frac{1}{p+1} & \text{if $a=0$}.
    \end{cases}\end{equation*}
We set $\mathbb{X}_n = \prod_{p \mid n} \mathbb{X}_p ^{\nu_p(n)}$, where $\nu_p(n)$ is the $p$-adic valuation of $n$.  The sequence $\mathbb{X}=\{\mathbb{X}_{n}\}_{n\in\mathbb{N}}$ was first introduced in \cite{GS} for the purpose of studying the distribution of the extreme values  of  $L(1,\chi_D)$ as $D$ varies over all fundamental discriminants. We denote the underlying probability measure on the sample space associated to $\mathbb{X}$ by $\mathbb{P}$.
For a representation $\pi$ as above and $t\in \mathbb{R}$, we associate the random series
$$-\ld(1+it,\pi, \mathbb{X})=\sum_{n=1}^{\infty}\frac{\Lambda(n)\lambda_{\pi}(n)\mathbb{X}_n}{n^{1+it}}.
$$

Our main theorem gives an upper bound on the discrepancy between the distribution of the random series $-\ld(1+it,\pi, \mathbb{X})$ and that of $-\frac{L'}{L}(1+it,\pi\otimes\chi_{D})$ as $D$ varies in $\mathcal{F}(N)$. Notice that if $\pi\cong\tilde{\pi}\otimes\left|\mathrm{det}\right|^{2it}$, then the distribution under consideration is 1-dimensional since the values involved are necessarily real. To see this, observe that for all primes $p$ and all $m\in\mathbb{N}$, we have 
\[\lambda_{\pi}(p^m)p^{-mit}=p^{-mit}\sum_{j=1}^d\alpha_{j,\pi}(p)^m=p^{-mit}\sum_{j=1}^d\left(\overline{\alpha_{j,\pi}(p)}p^{2it}\right)^m=\overline{\lambda_{\pi}(p^m)}p^{mit},\] which implies that $\lambda_{\pi}(p^m)p^{-mit}\in\mathbb{R}$. Otherwise, our distributions are 2-dimensional.  We denote the discrepancy by $\mathbb{D}_{1+it}\left(L_{\pi\otimes\chi_D}'/L_{\pi\otimes\chi_D};N\right)$ and define it as
\[\sup_{x\in \mathbb{R}}\left|\mathbb{P}_{N}\left(-\frac{L'}{L}(1+it,\pi\otimes\chi_{D})\leq x\right)-\mathbb{P}\left(-\ld(1+it,\pi,\mathbb{X})\leq x\right)\right|\] in the 1-dimensional case,  and
\[\sup_{\mathcal{R}\subset \mathbb{C}}\left|\mathbb{P}_{N}\left(-\frac{L'}{L}(1+it,\pi\otimes\chi_{D})\in \mathcal{R}\right)-\mathbb{P}\left(-\ld(1+it,\pi,\mathbb{X})\in \mathcal{R}\right)\right|,\]where $\mathcal{R}$ varies over all rectangles in $\mathbb{C}$ with sides parallel to the coordinate axes, in the 2-dimensional case.
\begin{thm}\label{main}
Let $\pi$ be a fixed cuspidal automorphic representation of $\mathrm{GL}_{d}(\mathbb{A}_{\mathbb{Q}})$ with unitary central character. Suppose that the Satake parameters of $\pi$ satisfy $\left|\alpha_{j,\pi}(p)\right|\leq p^{\theta}$ with $0\leq\theta<\frac14$ for all $j=1,2,\cdots,d$. Then we have \begin{align*}\mathbb{D}_{1+it}\left(L_{\pi\otimes\chi_D}'/L_{\pi\otimes\chi_D};N\right)\ll \frac{(\log\log N)^{2}}{\log N}.
\end{align*}

\end{thm}

It is clear from \eqref{RS-bound} that the above theorem holds if $d=1$. Moreover, it holds if $d=2$ since $\theta<\frac{7}{64}$ by \cite{K-S}. Indeed,  we get the following automorphic analogue of \cite[Theorem~1.1]{Lf} for logarithmic derivatives. 

\begin{cor}\label{GL2case}
Let $\pi$ be a fixed cuspidal automorphic representation of $\mathrm{GL}_{1}(\mathbb{A}_{\mathbb{Q}})$ or $\mathrm{GL}_{2}(\mathbb{A}_{\mathbb{Q}})$ with unitary central character.  Then we have
 \begin{align*}\mathbb{D}_{1+it}\left(L_{\pi\otimes\chi_D}'/L_{\pi\otimes\chi_D};N\right)&\ll \frac{(\log\log N)^{2}}{\log N}.
\end{align*}
\end{cor}

\begin{remarks} 
(i) We give the proof of Theorem \ref{main} only in the 2-dimensional case, i.e., $\pi\not\cong\tilde{\pi}\otimes\left|\mathrm{det}\right|^{2it}$. The proof in the 1-dimensional case follows analogous arguments. 

\noindent(ii) The condition on the bound for the Satake parameters in Theorem \ref{main} is needed in the proof of the exponential decay for the characteristic function of our random series (Proposition \ref{exponential-decay}) which is a crucial ingredient in our argument.  In some other parts of the argument, we only require Hypothesis H  (see Section \ref{section2.1}) which follows readily from the assumed bound on the Satake parameters in our main Theorem. 

\noindent(iii) Similar results can be obtained for the values of $\log{L}(1+it, \pi\otimes\chi_D)$ by following the proof of Theorem \ref{main}. 
 
 \noindent(iv) By examining the proof of Theorem \ref{main}, we can see that its assertion also holds when $1+it$ is replaced by $\sigma+it$ with $1-c_{\pi}<\sigma< 1$ for some constant $c_{\pi}>0$. The expected discrepancy bound in this case would be of size $O\left( \left(\frac{\log\log N}{\log N}\right)^\sigma\right)$. The lower bound on $\sigma$  is imposed by the zero density estimate used in the proof (see Lemma \ref{zero-density} and Proposition \ref{prop:short}) and the domain of convergence of $-\ld(s,\pi,\mathbb{X})$.
 
\noindent(v)  In some special cases, Corollary \ref{GL2case} holds with a discrepancy bound of size $O\left( \left(\frac{\log\log N}{\log N}\right)^\sigma\right)$ for a range of $\sigma$ that is wider than the one indicated in the previous remark when $1+it$ is replaced by $\sigma+it$. In fact, one can fix $\sigma$ to be much closer to $\frac12$ if a suitable zero density result is available for the family $L(s,\pi\otimes\chi_{D})$. For instance,  \cite[Theorem~3]{heath-brown} was used in \cite{Hamieh-Mcclenagan} to prove a discrepancy result for $-\frac{L'}{L}(\sigma,\chi_D)$ (which can be considered as a special case of Corollary \ref{GL2case}) that is valid for any $\frac12<\sigma<1$. Using the zero density theorem in \cite{Perelli-Pomykala}, one could achieve a similar result for $-\frac{L'}{L}(\sigma,f\otimes\chi_{D})$ when $f$ is primitive cusp form of weight 2.
\end{remarks}

We also use Theorem \ref{main} to derive an asymptotic bound for the small values of $\left|\frac{L'}{L}(1+it,\pi\otimes\chi_D)\right|$ when $\pi\cong\tilde{\pi}\otimes\left|\mathrm{det}\right|^{2it}$. The following result is an analogue of \cite[Corollary 1.4 ]{Hamieh-Mcclenagan}, and 
\cite[Theorem~1.1]{LL21} where the authors investigate the small values of $\left|\frac{L'}{L}(1,\chi)\right|$ for non-principal Dirichlet characters $\chi$ modulo $q$, as $q\to\infty$ over the primes.
\begin{thm}\label{min-values}
Let $t\in\mathbb{R}$ be fixed, and let $\pi$ be a cuspidal automorphic representation of $\mathrm{GL}_{d}(\mathbb{A}_{\mathbb{Q}})$ with unitary central character such that $\pi\cong\tilde{\pi}\otimes\left|\mathrm{det}\right|^{2it}$. Suppose that $\left|\alpha_{j,\pi}(p)\right|\leq p^{\theta}$ with $0\leq\theta<\frac14$ for all $j=1,2\cdots,d$. Let $\displaystyle{m_{N}=\min_{D\in\mathcal{F}(N)}\left(\left|\frac{L'}{L}(1+it,\pi\otimes\chi_D)\right|\right)}$. As $N\to\infty$, we have 
\[m_{N}\ll \frac{(\log\log N)^2}{\log N}.\]
More precisely, there are at least $N(\log\log{N})^2/\log{N}$ for which 
$$\left|\frac{L'}{L}(1+it,\pi\otimes\chi_D)\right|\ll \frac{(\log\log{N})^2}{\log{N}}.$$
\end{thm}

\begin{remark}
We derive our upper bound for the small values of $\left|\frac{L'}{L}(1+it,\pi\otimes\chi_D)\right|$ by an application of Theorem \ref{main} and using the positivity at the origin of the density function of the associated distribution. The assumption of these two facts together will result in 
\[m_{N}(1+it):=
\min_{D\in\mathcal{F}(N)}\left(\left|\frac{L'}{L}(1+it,\pi\otimes\chi_D)\right|\right)
\ll \mathbb{D}_{1+it}\left(L_{\pi\otimes\chi_D}'/L_{\pi\otimes\chi_D};N\right)  
\]
for 1-dimensional distributions and 
\[m_{N}(1+it)
\ll \sqrt{\mathbb{D}_{1+it}\left(L_{\pi\otimes\chi_D}'/L_{\pi\otimes\chi_D};N\right)}
\]
for 2-dimensional distributions. The restriction to $\pi\cong\tilde{\pi}\otimes\left|\mathrm{det}\right|^{2it}$  in Theorem \ref{min-values} is due to the fact that, in Lemma \ref{lem-positivity}, we are able to prove the positivity of the density function for $1$-dimensional distributions only.
\end{remark}
 The proof of Theorem \ref{min-values} is given in Section \ref{sec:min-values} as an application of Theorem \ref{main}. While our proof of Theorem \ref{main} follows the approach devised in \cite{Lf}, we deviate from their method in the last step of the proof. In order to avoid the need for a large deviation result for our family, we employ a 2-dimensional version of the classical Berry-Esseen inequality instead of the Beurling-Selberg functions used in \cite[Section~6]{Lf} to relate the distribution functions under consideration to their characteristic functions. In Section \ref{proof-main}, we show how this 2-dimensional Berry-Esseen inequality yields an upper bound for the discrepancy between $\mathbb{P}_{N}\left(-\frac{L'}{L}(1+it,\pi\otimes\chi_{D})\in \mathcal{R}\right)$ and $\mathbb{P}\left(-\ld(1+it,\pi,\mathbb{X})\right)$ in terms of the difference between their associated characteristic functions.  In doing so, we adapt some of the ideas outlined in \cite[Section~4.3]{M20}. In Section \ref{decay} we prove a rapid decay estimate for the characteristic function of the random series $-\ld(1+it,\pi,\mathbb{X})$ which  is crucial for applying the Berry-Esseen inequality. In Section \ref{char-function}, we prove Theorem \ref{FourierTransform} which shows that the characteristic function of the joint distribution of $\Re\left(-\frac{L'}{L}(1+it,\pi\otimes\chi_{D})\right)$ and $\Im\left(-\frac{L'}{L}(1+it,\pi\otimes\chi_{D})\right)$ can be very well approximated by the corresponding characteristic function of the random series $-\ld(1+it,\pi,\mathbb{X})$.  
The point of departure in the proof of Theorem \ref{FourierTransform} is a result asserting that  $-\frac{L'}{L}(1+it,\pi\otimes\chi_{D})$ can be approximated by a short Dirichlet polynomial  outside a set of fundamental discriminants $D$ of size $o(N)$. The proof of this approximation, although mostly standard, entails few complications arising from the fact that we do not assume the Generalized Ramanujan Conjecture in our work. We deal with these subtleties  by 
employing a truncated Perron's formula for automorphic $L$-functions \cite[Theorem~2.1]{L-Y}, a Brun-Titchmarsh type inequality \cite[Theorem~2.4]{S-T}, and  a recent zero density estimate \cite[Theorm~1.1]{H-T1}.  In view of this result, Theorem \ref{FourierTransform} can be extracted from a key result, Proposition \ref{prop2.3},  that compares the characteristic functions of short Dirichlet polynomials of the form \begin{equation}\label{partial-sums}\sum_{n\leq Y}\frac{\Lambda(n)\lambda_{\pi}(n)\chi_D(n)}{n^{1+it}}\quad\text{and}\quad\sum_{n\leq Y}\frac{\Lambda(n)\lambda_{\pi}(n)\mathbb{X}_n}{n^{1+it}}.\end{equation} In Section \ref{bridging}, we prove Proposition \ref{prop2.3}. More precisely, we show that \begin{align*}
&\frac{1}{\left|\mathcal{F}(N)\right|}\sum_{D\in\mathcal{F}(N)}\exp\left(z_1\sum_{n\leq Y}\frac{\Lambda(n)\lambda_{\pi}(n)\chi_{D}(n)}{n^{1+it}}+z_2\sum_{n\leq Y}\frac{\Lambda(n)\overline{\lambda_{\pi}(n)}\chi_{D}(n)}{n^{1-it}}\right)\end{align*}can be well approximated by \begin{align*}\mathbb{E}\left[\exp\left(z_{1}\sum_{n\leq Y}\frac{\Lambda(n)\lambda_{\pi}(n)\mathbb{X}_n}{n^{1+it}}+z_2\sum_{n\leq Y}\frac{\Lambda(n)\overline{\lambda_{\pi}(n)}\mathbb{X}_n}{n^{1-it}}\right)\right]\end{align*}for all complex numbers $z_1,z_2$ satisfying $|z_1|,|z_2|\ll\frac{\log N}{(\log\log N)^2}$.
We mention here that widening the range of $|z_1|$ and $|z_2|$ in Proposition \ref{prop2.3} would lead to an improvement in the discrepancy bound in Theorem \ref{main}. 
In view of the Taylor expansion of the exponential function, the proof of Proposition \ref{prop2.3} requires upper bounds of integral moments of the partial sums \eqref{partial-sums} which we establish in Section \ref{arithemtic-model} and Section \ref{random-model}.  


\subsection*{Conventions and Notation} Given two functions $f(x)$ and $g(x)$, we shall interchangeably use the notation  $f(x)=O(g(x))$ and $f(x) \ll g(x)$  to mean that there exists $M >0$ such that $|f(x)| \le M |g(x)|$ for all sufficiently large $x$. We write $f(x) \asymp g(x)$ to mean that the estimates $f(x) \ll g(x)$ and $g(x) \ll f(x)$ hold simultaneously.  Sometimes we will use the notation $f(x)\ll_t g(x)$, or alternatively $f(x)=O_t(g(x))$ to emphasize the dependence of the $O$-constant on the parameter $t$. Most of our $O$-constants depend on $\pi$ and $t$, although we sometimes drop the subscript to simplify the exposition. We use the notation $f(x)=o(g(x))$ if $\lim_{x\rightarrow \infty} f(x)/g(x)=0$.
Finally, the letter $p$ will always be used to denote a prime number.

\subsection*{Acknowledgements.} The authors would like to thank Jesse Thorner and Asif Zaman for useful correspondences related to this work.

\section{Preliminary Results}\label{arithemtic-model}

In this section, we introduce some notation and preliminary results pertaining to the arithmetic setting of automorphic $L$-functions, quadratic twists and logarithmic derivatives.
\subsection{Prime Number Sums}\label{section2.1}
The Generalized Ramanujan Conjecture (GRC) for a cuspidal automorphic representation $\pi$ of $\mathrm{GL}_{d}(\mathbb{A}_{\mathbb{Q}})$ asserts that $|\alpha_{j,\pi}(p)|\leq1$ for all $j=1,\dots,d$ and all primes $p$.  The following condition which follows from GRC is known as Hypothesis H  (see \cite[page~281]{R-S}). \\

\noindent{\bf Hypothesis H:} For any fixed $k\geq 2$, we have \[\sum_{p}\frac{(\log{p})^2\left|\lambda_{\pi}(p^k)\right|^2}{p^k}<\infty.\]  Observe that Hypothesis H holds if  $\left|\alpha_{j,\pi}(p)\right|\leq p^{\theta}$ with $0\leq\theta<\frac14$ for all $j=1,2,\cdots,d$, which is the assumption we make in Theorem \ref{main}.
In this work we make frequent use of the following prime number theorem for automorphic representations.
\begin{thm}\label{thm-PNT} Let $\pi$ and $\pi'$ be cuspidal automorphic representations of $\mathrm{GL}_{d}(\mathbb{A}_{\mathbb{Q}})$ and $\mathrm{GL}_{d'}(\mathbb{A}_{\mathbb{Q}})$ respectively, and assume that they satisfy Hypothesis $H$. Set $$\theta_{\pi,\pi'}(u)=\sum_{p\leq u}(\log p)\lambda_{\pi}(p)\overline{\lambda_{\pi'}(p)}.$$ 
Then for any $0<\epsilon<1$, we have
\begin{equation}\label{PNT}\theta_{\pi,\pi'}(u)=\begin{cases}\frac{u^{1+i\tau_0}}{1+i\tau_0}+O\left(\frac{u}{\log u}\right)&\;\text{if}\; \pi'\cong\pi\otimes\left|\mathrm{det}\right|^{i\tau_0}\;\text{for some}\; \tau_0\in\mathbb{R},\\ O_{\epsilon}\left(\frac{u}{(\log u)^{\frac{1-\epsilon}{dd'}}}\right)&\;\text{if}\;\pi'\not\cong\pi\otimes \left|\mathrm{det}\right|^{i\tau}\;\text{for any}\;\tau\in\mathbb{R}.\end{cases}\end{equation}


If $d,d'\leq 4$, then \eqref{PNT} is true without assuming Hypothesis $H$.
\end{thm}
\begin{remark} The first case of Theorem \ref{thm-PNT}  can be found in  \cite[Theorem~3]{Wu-Ye}. It follows from \cite[Theorem~2.3]{L-Y} and uses Hypothesis H to bound the contributions of composite prime powers in the sum $\sum_{n\leq u}\Lambda(n)\lambda_{\pi}(n)\overline{\lambda_{\pi'}(n)}$ by $O\left(\frac{u}{\log u}\right)$. The condition that at least one of $\pi$ and $\pi'$ is self dual in \cite{L-Y}  and \cite{Wu-Ye} can be removed by applying the recent zero-free region result of Humphries and Thorner \cite[Theorem~2.1]{H-T}. The second case of Theorem \ref{thm-PNT}  can be derived  from \cite[Theorem~2.6]{K-T} by using Hypothesis H to bound the contribution of composite prime powers. We note that  the exponent  $\frac{1-\epsilon}{dd'}$ can be replaced by 1 if $\pi$ or $\pi'$ is self-dual \cite[Theorem~3]{Wu-Ye}. \end{remark}
Using this theorem, we derive the following lemmas.
\begin{lem}\label{lem:pnt1H}
Assume Hypothesis H. As $X\to\infty$, we have 
\[\sum_{p>X}\frac{(\log p)^2}{p^{2}}|\lambda_{\pi}(p)|^2= \frac{\log X}{X}+O\left(\frac{1}{X}\right).\]
\end{lem}
\begin{proof} The proof follows from Theorem \ref{thm-PNT} and an application of Abel's summation formula.\end{proof}
\begin{lem}\label{lem:pnt2H}
Assume Hypothesis H. The following assertions hold.
\begin{enumerate}[label=(\roman*)]
\item If $\pi\cong\tilde{\pi}\otimes\left|\mathrm{det}\right|^{i\tau_0}$ for some $\tau_0\in\mathbb{R}$, then 
\[\sum_{p>X}\frac{(\log p)^2}{p^{2+2it}}\lambda_{\pi}(p)^2=\frac{\log X}{(1-2it+i\tau_0)X^{1-2it+i\tau_0}}+O\left(\frac{1}{X}\right),\] as $X\to\infty$.
\item If $\pi\not\cong\tilde{\pi}\otimes\left|\mathrm{det}\right|^{i\tau}$ for any $\tau\in\mathbb{R}$, then there exists $0<\alpha<1$ such that 
\[\sum_{p>X}\frac{(\log p)^2}{p^{2+2it}}\lambda_{\pi}(p)^2\ll \frac{(\log X)^{\alpha}}{X},\]
as $X\to\infty$.
\end{enumerate}
\end{lem}
\begin{proof}
Let $\pi_t=\pi\otimes \left|\mathrm{det}\right|^{-it}$ and $\tilde{\pi_t}$ be its dual representation. Then $\lambda_{\tilde{\pi_t}}(p)=\overline{\lambda_{\pi_t}(p)}=\overline{\lambda_{\pi}(p)}p^{it}$. We have 
\[\theta_{\pi_t,\tilde{\pi_t}}(u)=\sum_{p\leq u}(\log p)\lambda_{\pi_t}(p)\overline{\lambda_{\tilde{\pi_t}}(p)}=
\sum_{p\leq u}(\log p)\lambda_{\pi}(p)^2p^{-2it}.\]
 By Theorem \ref{thm-PNT},  for any $0<\epsilon<1$, we have
\[\theta_{\pi_t,\tilde{\pi_t}}(u)=\sum_{p\leq u}(\log p)\lambda_{\pi}(p)^2p^{-2it}=\begin{cases}\frac{u^{1+i\tau_0-2it}}{1+i\tau_0-2it}+O\left(\frac{u}{\log u}\right)&\;\text{if}\; \pi\cong\tilde{\pi}\otimes\left|\mathrm{det}\right|^{i\tau_0}\;\text{for some}\; \tau_0\in\mathbb{R},\\ O_{\epsilon}\left(\frac{u}{(\log u)^{\frac{1-\epsilon}{dd'}}}\right)&\;\text{if}\;\pi\not\cong\tilde{\pi}\otimes \left|\mathrm{det}\right|^{i\tau}\;\text{for any}\;\tau\in\mathbb{R}.\end{cases}\] 
If $\pi\cong\tilde{\pi}\otimes\left|\mathrm{det}\right|^{i\tau_0}$ for some $\tau_0\in\mathbb{R}$, we have
\begin{align*}\sum_{p>X}\frac{(\log p)^2}{p^{2+2it}}\lambda_{\pi}(p)^2&=\left[\frac{\log u}{u^2}\theta_{\pi_t,\tilde{\pi_t}}(u)\right]_{X}^{\infty}+\int_{X}^{\infty}\theta_{\pi_t,\tilde{\pi_t}}(u)\frac{2\log u-1}{u^3}\;du\\&=-\frac{\log X}{X^2}\left(\frac{X^{1+i\tau_0-2it}}{1+i\tau_0-2it}+O\left(\frac{X}{\log X}\right)\right)\\&\hspace{1em}+\int_{X}^{\infty}\left(\frac{u^{1+i\tau_0-2it}}{1+i\tau_0-2it}+O\left(\frac{u}{\log u}\right)\right)\frac{2\log u-1}{u^3}\;du
\\&=\frac{\log X}{(1+i\tau_0-2it)X^{1-i\tau_0+2it}}+O\left(\frac{1}{X}\right).\end{align*}
If $\pi\not\cong\tilde{\pi}\otimes\left|\mathrm{det}\right|^{i\tau}$, we have 
\begin{align*}\sum_{p>X}\frac{(\log p)^2}{p^{2+2it}}\lambda_{\pi}(p)^2&=\left[\frac{\log u}{u^2}\theta_{\pi_t,\tilde{\pi_t}}(u)\right]_{X}^{\infty}+\int_{X}^{\infty}\theta_{\pi_t,\tilde{\pi_t}}(u)\frac{2\log u-1}{u^3}\;du\\&=O\left(\frac{(\log X)^{\alpha}}{X}\right),\end{align*} where $\alpha$ can be taken to be $1-\frac{1-\epsilon}{dd'}$ for any $0<\epsilon<1$.
\end{proof}
\subsection{Rankin-Selberg $L$-functions} 
For a pair of automorphic representations $\pi$ and $\pi'$ of $\mathrm{GL}_{d}(\mathbb{A}_{\mathbb{Q}})$ and $\mathrm{GL}_{d'}(\mathbb{A}_{\mathbb{Q}})$ respectively, the associated Rankin-Selberg $L$-function is \[L(s,\pi\times\pi')=\prod_{p}\prod_{j=1}^d\prod_{j'=1}^{d'}\left(1-\frac{\alpha_{j,j',\pi\times\pi'}(p)}{p^s}\right)^{-1}=\sum_{n=1}^\infty\frac{a_{\pi\times\pi'}(n)}{n^s},\] where $\Re(s)>1$ and the parameters $\alpha_{j,j',\pi\times\pi'}(p)$ are indexed so that $\alpha_{j,j',\pi\times \pi^\prime}(p)=\alpha_{j,\pi}(p)\alpha_{j',\pi'}(p)$ for all but finitely many primes $p$.


We continue with the following two results from \cite{S-T}.
\begin{lem}
\cite[Lemma~2.2]{S-T}
\label{C-S}
Let $\pi$ be a cuspidal automorphic representation of $\mathrm{GL}_{d}(\mathbb{A}_{\mathbb{Q}})$. Then $$|\lambda_\pi(n)| \leq \sqrt{\lambda_{\pi \times \tilde{\pi}}(n)}\leq \frac{1}{2} (1+\lambda_{\pi \times \tilde{\pi}}(n)).$$
\end{lem}
\begin{lem}
\label{Sieve}\cite[Theorem~2.4]{S-T}
Let $\pi$ be a cuspidal automorphic representation of $\mathrm{GL}_{d}(\mathbb{A}_{\mathbb{Q}})$. If $Y\gg_d C(\pi \times \tilde{\pi})^{36d^2}$ and $1\leq H \leq Y^{\frac{1}{9d^2}}$, then
$$ \sum_{Y<n\leq Y e^{\frac{1}{H}}} \Lambda(n) \lambda_{\pi \times \tilde{\pi}}(n) \ll_d \frac{Y}{H}.$$
\end{lem}
Since $\lambda_{\pi \times \tilde{\pi}} (n) \geq 0$ (see \cite[page~318]{R-S}), then as a direct corollary of the above theorem we have
$$\sum_{Y <n\leq Y+\frac{Y}{H}} \Lambda(n) \lambda_{\pi \times \tilde{\pi}} (n) \leq \sum_{Y<n \leq Y e^{\frac{1}{H}}} \Lambda(n) \lambda_{\pi \times \tilde{\pi}} (n) \ll_d \frac{Y}{H}$$
under the conditions of Lemma \ref{Sieve} on $Y$ and $H$.

\subsection{Short Dirichlet polynomials}

In order to prove that the values $\frac{L'}{L}(1+it,\pi\otimes\chi_D)$ can be approximated by short Dirichlet polynomials  outside a set of fundamental discriminants $D$ of size $o(N)$, we require the following lemmas.

\begin{lem}
\label{Lp-L}
Let $T>1$, $\frac12<\sigma_0<1$, and $s=\sigma+it$. Suppose that $L(s, \pi\otimes\chi_D)$ has no zeros in the region $\sigma\geq \sigma_0$ and $|t|\leq T$. Then for all $\sigma\geq \sigma_0$ we have 
$$\frac{L^\prime}{L}(s, \pi\otimes \chi_D) \ll \frac{\log(D(|t|+2))}{\sigma-\sigma_0}.$$
\end{lem}

\begin{proof}

The result follows by adapting the proof of \cite[Lemma~2.2]{lamzouri3} to the setting of quadratic twists of automorphic $L$-functions.
\end{proof}

\begin{lem}\label{prop:ShortDirichlet}
Let $t\in \mathbb{R}$, $Y\gg_{\pi} 1$, $D\in \mathcal{F}(N)$, and $0<\delta\leq\frac{1}{3d^2}$.
Assume that 
$L(s, \pi\otimes\chi_D)$ has no zeros on $\Re(s)>1-\delta$ and $|\Im(s)| \leq Y^{\frac{1}{3d^2}}$. Then, we have
\begin{equation}
\label{pre-short}
-\frac{L^\prime}{L}(1+it, \pi\otimes \chi_D)= \sum_{n\leq Y} \frac{\Lambda(n)\lambda_\pi(n) \chi_D(n)}{n^{1+it}}+O\left(Y^{- \frac{\delta}{3}} (\log{N})\right).
\end{equation}
\end{lem}
\begin{proof}
By \cite[Theorem 2.1]{L-Y} and  for $c=1/\log{Y}$, $Y\geq 2$, $T\geq 2$, and $H\geq 2$, we have
\begin{equation}
\label{Perron}
\begin{split}
\sum_{n\leq Y} \frac{\Lambda(n)\lambda_\pi(n) \chi_D(n)}{n^{1+it}}&=\frac{1}{2\pi i} \int_{c-iT}^{c+iT} -\frac{L^\prime}{L}(w+1+it, \pi\otimes \chi_D) \frac{Y^w}{w} dw\\&\hspace{1em}+ O\left(\sum_{Y-\frac{Y}{H} <n \leq Y+\frac{Y}{H}} \frac{\Lambda(n) |\lambda_\pi(n)|}{n}\right)+O\left(\frac{H B(c)}{T}\right),
\end{split}
\end{equation}
where $B(c)=\sum_{n=1}^{\infty} \frac{\Lambda(n) |\lambda(n)|}{n^{c+1}}.$ 
By \cite[Formula (6.3)]{L-Y} we have $B(c)\ll \log{Y}$ and thus, the last error term in \eqref{Perron} is 
\begin{equation}
\label{l-error}
O\left( \frac{H \log{Y}}{T} \right).
\end{equation}

To handle the first error term in \eqref{Perron} observe that by Lemma \ref{C-S} we have
\begin{eqnarray*}
\sum_{Y-\frac{Y}{H} <n \leq Y+\frac{Y}{H}} \frac{\Lambda(n) |\lambda_\pi(n)|}{n} &\leq& \frac{1}{2} \left( \sum_{Y-\frac{Y}{H} <n \leq Y+\frac{Y}{H}} \frac{\Lambda(n) }{n}+\sum_{Y-\frac{Y}{H} <n \leq Y+\frac{Y}{H}} \frac{\Lambda(n) \lambda_{\pi\times \tilde{\pi}}(n)}{n}\right)\\
&\ll& \frac{1}{Y} \left(\sum_{Y-\frac{Y}{H} <n \leq Y+\frac{Y}{H}} {\Lambda(n) }+\sum_{Y-\frac{Y}{H} <n \leq Y+\frac{Y}{H}} {\Lambda(n) \lambda_{\pi\times \tilde{\pi}}(n)} \right).
\end{eqnarray*}
Now if $1\leq H \leq Y^{\frac{1}{9d^2}}$, then by Lemma \ref{Sieve} and the Brun-Titchmarsh inequality (e.g. see \cite[Theorem 6.6]{IK}) applied to the sums involving $\Lambda(n) \lambda_{\pi \times\tilde{\pi}}$ and $\Lambda(n)$ respectively, we have
\begin{equation}
\label{first error}
\sum_{Y-\frac{Y}{H} <n \leq Y+\frac{Y}{H}} \frac{\Lambda(n) |\lambda_\pi(n)|}{n} \ll \frac{1}{H}.
\end{equation}

Next we deal with the integral in \eqref{Perron}. For $\delta>0$ assume that $L(w, \pi\otimes \chi_D)$ does not have any zeros in the box
$$R_\delta:=\{(u, v);~1-\delta\leq u\leq 1,~|v|\leq T\}.$$
Letting $c_0=-\frac{\delta}{2}$ and applying the residue theorem yields
\begin{equation}
\label{residue}
\begin{split}
\frac{1}{2\pi}\int_{c-iT}^{c+iT}  -\frac{L^\prime}{L}(w+1+it, \pi\otimes \chi_D) \frac{Y^w}{w} dw&=- \frac{L^\prime}{L}(1+it, \pi\otimes \chi_D)\\&\hspace{1em}-\frac{1}{2\pi i} \left(\int_{c-iT}^{c_0+iT} (\cdot)-  \int_{c_0-iT}^{c+iT} (\cdot)- \int_{c_0+it}^{c_0-iT}(\cdot)\right).
\end{split}
\end{equation}
We now estimate the integrals in \eqref{residue}. By employing Lemma \ref{Lp-L} we have
\begin{equation}\label{vertical}
\begin{split}
\left| \int_{c_0+iT}^{c_0-iT} -\frac{L^\prime}{L}(w+1+it, \pi\otimes \chi_D) \frac{Y^w}{w} dw \right|&\ll \int_{-T}^{T}    \frac{\log(|D|(|v|+2))}{(1-\frac{\delta}{2})-(1-\delta)} \frac{Y^{-\frac{\delta}{2}}}{\sqrt{v^2+\frac{\delta^2}{4}}} dv \\
&\ll_\delta Y^{-\frac{\delta}{2}} \log(|D|(T+2)) \log(T+2).
\end{split}
\end{equation}
For the two remaining integrals on the right-hand side of \eqref{residue}, by using Lemma \ref{Lp-L}, we have

\begin{equation}\label{horizontal}
\begin{split}
\left| \int_{c+iT}^{c_0+iT} -\frac{L^\prime}{L}(w+1+it, \pi\otimes \chi_D) \frac{Y^w}{w} dw \right|&\ll \int_{-\frac{\delta}{2}}^{1/\log{Y}}    \frac{\log(|D|(|T|+2))}{(1+u)-(1-\delta)} \frac{Y^{u}}{\sqrt{u^2+T^2}} du\\
&\ll_\delta \frac{ \log(|D|(T+2))}{T\log{Y}}.
\end{split}
\end{equation}

Thus, from \eqref{Perron}, \eqref{l-error}, \eqref{first error}, \eqref{residue}, \eqref{vertical}, and \eqref{horizontal} and under the assumption that $L(w, \pi \times \chi_D)$ does not have any zero in $R_\delta$, we get
\begin{equation}
\label{final}
\begin{split}
-\frac{L^\prime}{L}(1+it, \pi\otimes \chi_D)&= \sum_{n\leq Y} \frac{\Lambda(n)\lambda_\pi(n) \chi_D(n)}{n^{1+it}}+O\left(  \frac{H\log{Y}}{T} \right)+O\left( \frac{1}{H} \right)\\&+O\left(Y^{-\frac{\delta}{2}} \log(|D|(T+2)) \log(T+2) \right)+ O\left( \frac{ \log(|D|(T+2))}{T\log{Y}} \right),
\end{split}
\end{equation}
 for $1\leq H \leq Y^{\frac{1}{9d^2}}$. The result follows by setting $H=Y^{\frac{1}{9d^2}}$ and  $T=Y^\frac{1}{3d^2}$ in \eqref{final}. 
\end{proof}
We also require the following zero-density estimate which is a direct application of \cite[Theorem~1.1]{H-T1}.

\begin{lem}\label{zero-density}  Let $\pi$ be a cuspidal automorphic representation of $\mathrm{GL}_{d}(\mathbb{A}_{\mathbb{Q}})$ with unitary central character, and let $C(\pi)$ be the analytic conductor of $\pi$ (as defined in \cite[page 95]{IK}). 
Let $T,N\geq1$, and set \[N_{\pi \otimes \chi_D}(\sigma, T)=\left|\left\{\rho=\beta+i\gamma;~L(\rho, \pi\otimes \chi_D)=0,~\beta \geq \sigma,~|\gamma|\leq T\right\}\right|.\]  For $\epsilon>0$, we have  
\begin{equation}
\label{ZD}
\sum_{D\in\mathcal{F}(N)}N_{\pi \otimes \chi_D}(\sigma, T)\ll_{\epsilon,d}(C(\pi)NT)^{18d(1-\sigma)+\epsilon},
\end{equation}
provided that $\frac12\leq\sigma\leq1$.
\end{lem}

\begin{remark}
An application of \cite[Theorem~1.1]{H-T1} will give a result similar to \eqref{ZD} for $N_{\pi \times \chi_D}(\sigma, T)$ attached to the zeros of the Rankin-Selberg $L$-function $L(s, \pi\times\chi_D)$.
We know that the local $L$-functions $L_p(s, \pi\times\chi_D)$ and $L_p(s, \pi\otimes\chi_D)$ coincide for primes $p\nmid (q_\pi, D)$, where $q_\pi$ is the conductor of $\pi$. Moreover, the local parameters $\alpha_{j, \pi\times\chi_D}$ (for $1\leq j \leq d$) satisfy the bound  \eqref{RS-bound}. Therefore, $N_{\pi \times \chi_D}(\sigma, T)=N_{\pi \otimes \chi_D}(\sigma, T)$ for $\sigma>\tfrac{1}{2}-\tfrac{1}{d^2+1}$ and thus $\eqref{ZD}$ holds.

\end{remark}

 We finally arrive at the desired approximation.

\begin{prop}\label{prop:short}
There are positive constants $\delta_\pi$ and $\eta_\pi$ (depending only on $\pi$) such that for all but $O(N^{\frac34})$ fundamental discriminants $D$ in $\mathcal{F}(N)$, we have 
\begin{equation}\label{eqn:short}
-\frac{L'}{L}(1+it,\pi\otimes\chi_D)=\sum_{n\leq Y}\frac{\Lambda(n)\lambda_{\pi}(n)\chi_D(n)}{n^{1+it}} +O(Y^{-\delta_{\pi}}),
\end{equation}
whenever $ (\log{N})^{\eta_\pi}\ll Y\ll N^{3d^2}$. 
\end{prop}
%
\begin{proof}
Let $T=Y^{\frac{1}{3d^2}}$ where $Y\ll N^{3d^2}$, $0<\delta<\min\{ \frac{1}{144d}, \frac{1}{3d^2}\}$, and $\epsilon=\frac14$. Then as a direct corollary of Lemmas \ref{prop:ShortDirichlet} and \ref{zero-density} we conclude that \eqref{pre-short} holds for all but $O(N^\frac{3}{4})$ fundamental discriminants in $\mathcal{F}(N)$.
\end{proof}

In the rest of the paper, we shall denote by  $\mathcal{A}(N)$  the subset of $\mathcal{F}(N)$ for which \eqref{eqn:short} holds for some $\delta_\pi$ and $\eta_\pi$. We also define $\mathcal{E}(N)$ by writing $\mathcal{A}(N)=\mathcal{F}(N)\setminus\mathcal{E}(N)$.

\subsection{More Lemmas} In what follows, we compute upper bounds for $2k$-th moments of sums associated with the short Dirichlet polynomials appearing in Proposition \ref{prop:short}. 

\begin{lem}\label{lem3.2}
Let $2\leq y \leq z$. Then, uniformly for $k\leq \frac{{\log{N}}}{{{6}\log{z}}}$, we have
$$\frac{1}{N} \sum_{D\in \mathcal{F}(N)} \left|\sum_{y\leq p\leq z} \frac{(\log p)\lambda_\pi(p) \chi_D(p)}{p^{1+it}}     \right|^{2k}\ll k! \left(\sum_{y\leq p \leq z}  \frac{(\log{p})^2|\lambda_\pi(p)|^2}{p^2}   \right)^k.$$
\end{lem}
\begin{proof}
The proof closely follows \cite[Lemma~3]{S}. We have
$$\left(\sum_{y\leq p\leq x} \frac{(\log p)\lambda_\pi(p) \chi_D(p)}{p^{1+it}} \right)^k=\sum_{y^k\leq n \leq z^k}  \frac{a_{k, y, z}(n)}{n^{1+it}}, $$
where
$$a_{k, y, z}(n) =\begin{cases}
{{k}\choose{\alpha_1, \cdots, \alpha_r}} \prod_{i=1}^{r} ((\log p)\lambda_\pi(p)\chi_D(p))^{\alpha_i} & \text{if}~ n=p_1^{\alpha_1}\cdots p_r^{\alpha_r},~p_i's~ \text{distinct},~y\leq p_i \leq z,\\
0&\text{otherwise}.
\end{cases}$$
Thus, we get
\begin{equation}
\label{Polya}
\frac{1}{N} \sum_{D\in \mathcal{F}(N)} \left|\sum_{y\leq p\leq z} \frac{(\log p)\lambda_\pi(p) \chi_D(p)}{p^{1+it}}     \right|^{2k}= \sum_{y^k \leq m, n \leq z^k} \frac{a_{k, y, z}(m) \overline{a_{k, y, z}(n)} }{(mn)^{1+it}}\left( \frac{1}{N} \sum_{D\in \mathcal{F}(N)} \chi_D(mn)\right).
\end{equation}
From \cite[Lemma~4.1]{GS} we know that if $mn$ is not a perfect square, then
\begin{equation}\label{eqn:Polya-Vinog}\sum_{D\in \mathcal{F}(N)} \chi_D(mn) \ll N^{\frac{1}{2}} (mn)^{\frac{1}{4}} (\log{(mn}))^{\frac{1}{2}}.\end{equation}
Applying  this upper bound in \eqref{Polya} yields
\begin{equation}\label{inequality}
\begin{split}
\frac{1}{N} \sum_{D\in \mathcal{F}(N)} \left|\sum_{y\leq p\leq z} \frac{(\log p)\lambda_\pi(p) \chi_D(p)}{p^{1+it}}     \right|^{2k}&\ll \sum_{y^k \leq  n \leq z^k} \frac{|a_{k, y, z}(n)|^2}{n^2}\\&\hspace{1em}+N^{-\frac{1}{2}} \sum_{\substack{y^k \leq m, n \leq z^k\\ mn\neq \square}} \frac{|a_{k, y, z}(m)||a_{k, y, z}(n)|}{(mn)^{\frac{3}{4}}} (\log(mn))^{\frac{1}{2}}.
\end{split}
\end{equation}
Observe that $$2\frac{|a_{k, y, z}(m)||a_{k, y, z}(n)|}{(mn)^{\frac{3}{4}}}\leq \frac{|a_{k, y, z}(m)|^2}{m^{\frac{3}{2}}}+ \frac{|a_{k, y, z}(n)|^2}{n^{\frac{3}{2}}}.$$
By application of this inequality in the last term of \eqref{inequality}, we have
\begin{equation}
\label{inequality-2}
\begin{split}
\frac{1}{N} \sum_{D\in \mathcal{F}(N)} \left|\sum_{y\leq p\leq z} \frac{(\log p)\lambda_\pi(p) \chi_D(p)}{p^{1+it}} \right|^{2k}
&\ll \sum_{y^k \leq  n \leq z^k} \frac{|a_{k, y, z}(n)|^2}{n^2}\\&\hspace{1em}+N^{-\frac{1}{2}} \sum_{y^k \leq n \leq z^k} \frac{|a_{k, y, z}(n)|^2}{n^{\frac{3}{2}}} \sum_{\substack{y^k \leq m \leq z^k\\mn\neq \square}}(\log(mn))^{\frac{1}{2}}.\end{split}
\end{equation}
We deduce from \eqref{inequality-2} that
\begin{equation}
\label{inequality-3}
\begin{split}
\frac{1}{N} \sum_{D\in \mathcal{F}(N)} \left|\sum_{y\leq p\leq z} \frac{(\log p)\lambda_\pi(p) \chi_D(p)}{p^{1+it}}     \right|^{2k}
&\ll \sum_{y^k \leq  n \leq z^k} \frac{|a_{k, y, z}(n)|^2}{n^2}\\&\hspace{1em}+N^{-\frac{1}{2}} z^{\frac{3k}{2}} (\log(z^{2k}))^{\frac{1}{2}} \sum_{y^k \leq n \leq z^k} \frac{|a_{k, y, z}(n)|^2}{n^{{2}}}\\&\ll
\sum_{y^k \leq  n \leq z^k} \frac{|a_{k, y, z}(n)|^2}{n^2}
\end{split}
\end{equation}
since $k\leq (\log{N})/(6\log{z})$. The desired result follows from \eqref{inequality-3} since by an argument similar to the one given in the proof of \cite[Lemma~3]{S} we have
$$\sum_{y^k \leq  n \leq z^k} \frac{|a_{k, y, z}(n)|^2}{n^2}\leq k! \left(\sum_{y\leq p \leq z}  \frac{(\log{p})^2|\lambda_\pi(p)|^2}{p^2}   \right)^k.$$
\end{proof}

\begin{lem}\label{lem3.3}
Let $A\geq1$ be fixed and set $Y=(\log N)^A$. Let $k$ be an integer satisfying $2\leq k\leq \frac{\log N}{6A\log\log N}$. Under the assumption of Hypothesis H, we have 
\begin{equation*}
\frac{1}{N}\sum_{D\in\mathcal{F}(N)}\left|\sum_{n\leq Y}\frac{\Lambda(n)\lambda_{\pi}(n)\chi_{D}(n)}{n^{1+it}}\right|^{2k}\ll\left(C\log k\right)^{2k}
\end{equation*}
for some positive constant $C$ that depends only on $\pi$.
\end{lem}
\begin{proof}
We have
 \begin{align*}
 \frac{1}{N}\sum_{D\in\mathcal{F}(N)}\left|\sum_{n\leq Y}\frac{\Lambda(n)\lambda_{\pi}(n)\chi_{D}(n)}{n^{1+it}}\right|^{2k}&=\frac{1}{N}\sum_{D\in\mathcal{F}(N)}\Bigg|\sum_{p\leq \tfrac{k}{\log k}}\frac{(\log p)\lambda_{\pi}(p)\chi_{D}(p)}{p^{1+it}}\\&\hspace{3em}+\sum_{\tfrac{k}{\log k}<p\leq Y}\frac{(\log p)\lambda_{\pi}(p)\chi_{D}(p)}{p^{1+it}}+\sum_{\substack{n\geq2\\ p^n\leq Y}}\frac{(\log p)\lambda_{\pi}(p^n)\chi_{D}(p^n)}{p^{n+int}}\Bigg|^{2k}.\end{align*}
It follows that 
 \begin{align*}
 \frac{1}{N}\sum_{D\in\mathcal{F}(N)}\left|\sum_{n\leq Y}\frac{\Lambda(n)\lambda_{\pi}(n)\chi_{D}(n)}{n^{1+it}}\right|^{2k}&\leq \frac{9^k}{N}\sum_{D\in\mathcal{F}(N)}\left|\sum_{p\leq \tfrac{k}{\log k}}\frac{(\log p)\lambda_{\pi}(p)\chi_{D}(p)}{p^{1+it}}\right|^{2k}\\&\hspace{1em}+ \frac{9^k}{N}\sum_{D\in\mathcal{F}(N)}\left|\sum_{\tfrac{k}{\log k}<p\leq Y}\frac{(\log p)\lambda_{\pi}(p)\chi_{D}(p)}{p^{1+it}}\right|^{2k}\\&\hspace{1em}+ \frac{9^k}{N}\sum_{D\in\mathcal{F}(N)}\left|\sum_{\substack{n\geq2\\ p^n\leq Y}}\frac{(\log p)\lambda_{\pi}(p^n)\chi_{D}(p^n)}{p^{n+int}}\right|^{2k}.
 \end{align*}

By Lemma \ref{lem3.2}, we know that 
\begin{align*}
\frac{1}{N} \sum_{D\in \mathcal{F}(N)} \left|\sum_{\frac{k}{\log k}\leq p\leq Y} \frac{(\log p)\lambda_\pi(p) \chi_D(p)}{p^{1+it}}     \right|^{2k}\ll k! \left(\sum_{\frac{k}{\log k}\leq p \leq Y}  \frac{(\log{p})^2|\lambda_\pi(p)|^2}{p^2}   \right)^k.
\end{align*}
Hence,
\begin{equation}\label{eqn1:lem3.3}
\begin{split}
\frac{1}{N}\sum_{D\in\mathcal{F}(N)}\left|\sum_{n\leq Y}\frac{\Lambda(n)\lambda_{\pi}(n)\chi_{D}(n)}{n}\right|^{2k}&\ll 9^k\left(\sum_{p\leq \frac{k}{\log k}}\frac{(\log p)\left|\lambda_{\pi}(p)\right|}{p}\right)^{2k}\\&\hspace{1em}+9^{k}k! \left(\sum_{\frac{k}{\log k}\leq p \leq Y}  \frac{(\log{p})^2|\lambda_\pi(p)|^2}{p^2}  \right)^k \\&\hspace{1em}+9^k\left(\sum_{\substack{n\geq2\\ p^n\leq Y}}\frac{(\log p)\left|\lambda_{\pi}(p^n)\right|}{p^n}\right)^{2k}.
\end{split}
\end{equation}

We have 
\begin{align*}
\left(\sum_{p\leq \frac{k}{\log k}}\frac{(\log p)\left|\lambda_{\pi}(p)\right|}{p}\right)^{2k}&\leq \left(\sum_{p\leq \frac{k}{\log k}}\frac{\log p}{p}\right)^{k}\left(\sum_{p\leq \frac{k}{\log k}}\frac{(\log p)|\lambda_{\pi}(p)|^2}{p}\right)^k\leq (C_1\log(k/\log k))^{2k},
\end{align*}
where $C_1$ is a positive constant that depends only on $\pi$, and for the last inequality we use \cite[page~150]{A-S} to bound $\sum_{p\leq\frac{k}{\log k}}\frac{(\log p)\left|\lambda_{\pi}(p)\right|^2}{p}$.

Now Abel's summation formula yields
\begin{align*}
 \sum_{\frac{k}{\log k}\leq p \leq Y}  \frac{(\log{p})^2|\lambda_\pi(p)|^2}{p^2}  &=A(Y)\frac{\log Y}{Y^2}-A(k/\log k)\frac{\log(k/\log k)}{(k/\log k)^2}-\int_{\frac{k}{\log k}}^{Y} A(t)\left(\frac{1-2\log t}{t^3}\right)\;dt,
\end{align*}
where $A(t)=\sum_{p\leq t}\log p\left|\lambda_{\pi}(p)\right|^2$. Since $A(t)\ll t$ (see \cite[page~150]{A-S}), with our choices for $k$ and $Y$, we have
\begin{align*}
 \sum_{\frac{k}{\log k}\leq p \leq Y}  \frac{(\log{p})^2|\lambda_\pi(p)|^2}{p^2} \leq C_2\frac{\log(k/\log k)}{k/\log k},
\end{align*}
for some positive constant $C_2$ that depends only on $\pi$. Hence, the second sum on the RHS of \eqref{eqn1:lem3.3} is bounded by 
\[\frac{k!}{k^k}\left(9C_2\log k\log(k/\log k)\right)^{k}\leq \left(9C_2\log^2 k\right)^{k}.\]
For the last sum in \eqref{eqn1:lem3.3}, we use the inequality \begin{equation}\label{eqn:lambda-bound}|\lambda_{\pi}(n)|\leq \frac{1}{2}(1+|\lambda_{\pi}(n)|^2)\end{equation} and \cite[Equation~(2)]{A-S} to get
\begin{align*}\sum_{\substack{n\geq2\\ p^n\leq Y}}\frac{(\log p)\left|\lambda_{\pi}(p^n)\right|}{p^n}&\leq\frac{1}{2}\sum_{\substack{n\geq2\\ p^n\leq Y}}\frac{\log p}{p^n}+\frac{1}{2}\sum_{\substack{n\geq2\\ p^n\leq Y}}\frac{(\log p)\left|\lambda_{\pi}(p^n)\right|^2}{p^n}\\&=O(1).
\end{align*}
In this argument, Hypothesis H is required for the application of \cite[Equation~(2)]{A-S}. 
Combining all these estimates gives the desired result.
\end{proof}
We end this section by providing an upper bound for the second moment of the values $\frac{L'}{L}(1+it,\pi\otimes\chi_D)$ as $D$ varies in the set $\mathcal{A}(N)$ which was introduced after Proposition \ref{prop:short}. This result is used in the proof of Theorem \ref{main} in Section \ref{proof-main}.

\begin{lem}\label{2ndmoment} Assume Hypothesis $H$.  As $N\to\infty$, we have
\[\frac{1}{|\mathcal{A}(N)|}\sum_{D\in\mathcal{A}(N)}\left|\frac{L'}{L}(1+it,\pi\otimes\chi_D)\right|^2\ll_{\pi} 1.\] 
\end{lem}
\begin{proof}
Let $D\in\mathcal{A}(N)$. By Proposition \ref{prop:short}, there exist $\delta_{\pi}, \eta_{\pi}>0$ such that 
\begin{equation}\label{eqn:L'/L2}
\left|\frac{L'}{L}(1+it,\pi\otimes\chi_D)\right|^2=\left|\sum_{n\leq Y}\frac{\Lambda(n)\lambda_{\pi}(n)\chi_D}{n^{1+it}} \right|^2+ O\left(Y^{-\delta_{\pi}}\sum_{n\leq Y}\frac{\Lambda(n)\left|\lambda_{\pi}(n)\right|}{n} \right)+O(Y^{-2\delta_{\pi}}),
\end{equation}
provided that $(\log N)^{\eta_{\pi}}\ll Y\ll N^{3d^2}$.
It follows from \eqref{eqn:lambda-bound} and \eqref{eqn:L'/L2} that
\begin{align*}
\left|\frac{L'}{L}(1+it,\pi\otimes\chi_D)\right|^2&=\left|\sum_{p\leq Y}\frac{(\log p)\lambda_{\pi}(p)\chi_D(p)}{p^{1+it}}+\sum_{\substack{p^n\leq Y\\n\geq2}}\frac{(\log p)\lambda_{\pi}(p^n)\chi_D(p^n)}{p^{n+nit}} \right|^2\\&\hspace{2em}+ O\left(Y^{-\delta_{\pi}}\sum_{n\leq Y}\frac{\Lambda(n)}{n} \right)+O\left(Y^{-\delta_{\pi}}\sum_{n\leq Y}\frac{\Lambda(n)\left|\lambda_{\pi}(n)\right|^2}{n} \right)+O(Y^{-2\delta_{\pi}}).\end{align*}
Since $\sum_{n\leq Y}\frac{\Lambda(n)}{n}$ and $\sum_{n\leq Y}\frac{\Lambda(n)\left|\lambda_{\pi}(n)\right|^2}{n}$ are both $O(\log Y)$ by Mertens' bound and \cite[Equation~(3)]{A-S} respectively, we get
\begin{align*}\left|\frac{L'}{L}(1+it,\pi\otimes\chi_D)\right|^2&\ll \left|\sum_{p\leq Y}\frac{(\log p)\lambda_{\pi}(p)\chi_D(p)}{p^{1+it}}\right|^2+\left|\sum_{\substack{p^n\leq Y\\n\geq2}}\frac{(\log p)\lambda_{\pi}(p^n)\chi_D(p^n)}{p^{n+nit}} \right|^2+O(Y^{-\nu_{\pi}}),
\end{align*}
for some $\nu_{\pi}>0$ depending on $\pi$. Assuming Hypothesis H, by \eqref{eqn:lambda-bound} we have
\begin{align*}
\left|\sum_{\substack{p^n\leq Y\\n\geq2}}\frac{(\log p)\lambda_{\pi}(p^n)\chi_D(p^n)}{p^{n+nit}} \right|^2&\ll\left(\sum_{\substack{p^n\leq Y\\n\geq2}}\frac{(\log p)\left|\lambda_{\pi}(p^n)\right|}{p^{n}} \right)^2\\&\ll\left(\sum_{\substack{p^n\leq Y\\n\geq2}}\frac{\log p}{p^{n}} +\sum_{\substack{p^n\leq Y\\n\geq2}}\frac{(\log p)\left|\lambda_{\pi}(p^n)\right|^2}{p^{n}}\right)^2\\&\ll_{\pi} 1.
\end{align*}
Hence, under Hypothesis H we have
\begin{equation*}
\left|\frac{L'}{L}(1+it,\pi\otimes\chi_D)\right|^2\ll  \left|\sum_{p\leq Y}\frac{(\log p)\lambda_{\pi}(p)\chi_D(p)}{p^{1+it}}\right|^2 +O_{\pi}(1).
\end{equation*}

Taking the average over $\mathcal{A}(N)$ gives
\begin{align*}
\frac{1}{|\mathcal{A}(N)|}\sum_{D\in\mathcal{A}(N)}\left|\frac{L'}{L}(1+it,\pi\otimes\chi_D)\right|^2&\ll\frac{1}{|\mathcal{A}(N)|}\sum_{D\in\mathcal{A}(N)}\left|\sum_{p\leq Y}\frac{(\log p)\lambda_{\pi}(p)\chi_D(p)}{p^{1+it}}\right|^2 +O_{\pi}(1)
\\&\ll \sum_{p\leq Y}\frac{(\log p)^2|\lambda_{\pi}(p)|^2}{p^2} +O_{\pi}(1).
\end{align*}
Since, by \eqref{RS-bound}, $|\lambda_{\pi}(p)|\ll p^{\theta}$ with $0\leq\theta<\frac12$, from the above we get
\begin{align*}
\frac{1}{|\mathcal{A}(N)|}\sum_{D\in\mathcal{A}(N)}\left|\frac{L'}{L}(1+it,\pi\otimes\chi_D)\right|^2\ll_{\pi}1\end{align*}
as desired.
\end{proof}

\section{The Random Model}\label{random-model}

Recall the definition of the random model $\mathbb{X}=\{\mathbb{X}_n\}_{n\in\mathbb{N}}$ given by  \[\mathbb{X}_n = \prod_{p \mid n} \mathbb{X}_p ^{\nu_p(n)},\] where $\nu_p(n)$ is the $p$-adic valuation of $n$, and $\{\mathbb{X}_p\}_{p\;\text{prime}}$ is the sequence of independent random variables given by
	\begin{equation}\label{eqn:random-model}\mathbb{P} \big( \mathbb{X}_p = a \big) = \begin{cases}
        \frac{p}{2(p+1)} & \text{if $a= \pm 1$}, \\
        \frac{1}{p+1} & \text{if $a=0$}.
    \end{cases}\end{equation}
The random variables $\mathbb{X}_n$ satisfy
\begin{equation}\label{eq-E(X(n))}
    \mathbb{E} \left[ \mathbb{X}_n \right] = 
    \begin{cases}
        \prod_{p \mid n} \Big( \frac{p}{p+1} \Big) &\text{if $n$ is a square,} \\
        0 &\text{otherwise}.
    \end{cases}
\end{equation}

The random sum $-\ld(1+it,\pi, \mathbb{X})$ given by
\begin{equation}\label{eqn:random-sum1}
-\ld(1+it,\pi, \mathbb{X})=\sum_{n=1}^{\infty}\frac{\Lambda(n)\lambda_{\pi}(n)\mathbb{X}_n}{n^{1+it}}
\end{equation}
can be written as $\sum_{p}\frac{(\log p)\lambda_{\pi}(p)\mathbb{X}_p}{p^{1+it}} +O(1)$, and the latter sum is almost surely convergent by the Menshov-Rademacher Theorem (see for example \cite[Proposition~B.10.5]{kowalski}). 
Moreover, the random sum
\begin{equation}\label{eqn:random-sum2}
\sum_{j=1}^d\sum_{p}\frac{\alpha_{j,\pi}(p)\mathbb{X}_{p}\log p}{p^{1+it}-\alpha_{j,\pi}(p)\mathbb{X}_p}
\end{equation}
is almost surely convergent by Kolmogorov's Theorem (see for example \cite[Proposition~B.10.1]{kowalski}). 
More generally, let $\tau>1-\frac{1}{d^2+1}$, and let $U_{\tau}=\{s\in\mathbb{C};\Re(s)>\tau\}$. It follows from the Menshov-Rademacher theorem that the random series \begin{equation}\label{eqn:random-sum-s-1}\sum_{n=1}^\infty \frac{\Lambda(n)\lambda_{\pi}(n)\mathbb{X}_n}{n^{s}}\end{equation} is almost surely convergent on $U_{\tau}$, and so it defines an almost surely holomorphic function there. We also consider the random series \begin{equation}\label{eqn:random-sum-s-2}\sum_{j=1}^d\sum_{p}\frac{(\log p)\alpha_{j,\pi}(p)\mathbb{X}_{p}}{p^{s}-\alpha_{j,\pi}(p)\mathbb{X}_p}, \end{equation}
which, by Kolmogorov's  theorem, 
is almost surely convergent  on $U_{\tau}$, and so it defines a holomorphic function there. One could easily verify that the series \eqref{eqn:random-sum-s-1} and \eqref{eqn:random-sum-s-2}  are equal for all $s$ with $\Re(s)>1$. By analytic continuation, we see that
\begin{equation*}\label{eqn:two-series}\sum_{n=1}^\infty \frac{\Lambda(n)\lambda_{\pi}(n)\mathbb{X}_n}{n^{s}}=\sum_{j=1}^d\sum_p \frac{(\log p)\alpha_{j,\pi}(p)
 \mathbb{X}_p}{p^{s} -\alpha_{j,\pi}(p)\mathbb{X}_p}\end{equation*}
almost surely in $U_{\tau}$. 
In particular, we have 
\begin{equation*}\label{eqn:L'/L-random}
  -\ld(1+it,\pi,\mathbb{X})=\sum_{n=1}^{\infty}\frac{\Lambda(n)\lambda_{\pi}(n)\mathbb{X}_n}{n^{1+it}}=\sum_{j=1}^d\sum_{p}\frac{(\log p)\alpha_{j,\pi}(p)\mathbb{X}_{p}}{p^{1+it}-\alpha_{j,\pi}(p)\mathbb{X}_p}.
\end{equation*}

In what follows we will be considering the partial random sums 
\[\sum_{n\leq Y}\frac{\Lambda(n)\lambda_{\pi}(n)\mathbb{X}_n}{n^{1+it}}.\]
We need the following lemmas involving sums of the above form.
\begin{lem}\label{lem3.2-random}
Let $2\leq y \leq z$. Then, uniformly for any positive integer $k$ we have
\[\mathbb{E}\left[ \left|\sum_{y\leq p\leq z} \frac{(\log p)\lambda_\pi(p) \mathbb{X}_p}{p^{1+it}}     \right|^{2k}\right]\ll k! \left(\sum_{y\leq p \leq z}  \frac{(\log{p})^2|\lambda_\pi(p)|^2}{p^2}   \right)^k.\]
\end{lem}
\begin{proof}
The proof of this lemma is similar to the proof of Lemma \ref{lem3.2} where we use \eqref{eq-E(X(n))} in lieu of the Polya-Vinogradov type inequality \eqref{eqn:Polya-Vinog}.
\end{proof}

\begin{lem}\label{lem3.3-random}
Let $A\geq1$ be fixed and set $Y=(\log N)^A$. Let $k\geq2$ be any integer. Under Hypothesis H, we have 
\begin{equation*}
\mathbb{E}\left[\left|\sum_{n\leq Y}\frac{\Lambda(n)\lambda_{\pi}(n)\mathbb{X}_n}{n^{1+it}}\right|^{2k}\right]\ll\left(C\log k\right)^{2k},
\end{equation*}
for some positive constant $C$ that depends only on $\pi$.
\end{lem}
\begin{proof}
This lemma follows an argument similar to the one used in the proof of Lemma \ref{lem3.3}. We use Lemma \ref{lem3.2-random} in lieu of Lemma \ref{lem3.2}.
\end{proof}

 \begin{lem}\label{lem4.1}
Suppose that  $|\lambda_{\pi}(p)|\ll p^{\theta}$ for some $0\leq\theta<\frac12$. Let $0<\epsilon<\frac12-\theta$ be given. Then if $u$ and $v$ are real numbers such that $|u|+|v|\leq Y^{\frac12-(\theta+\epsilon)}$, we have
\begin{align*}
&\mathbb{E}\left[\exp\left(iu\Re\left(-\ld(1+it,\pi, \mathbb{X})\right)+iv\Im\left(-\ld(1+it,\pi, \mathbb{X})\right)\right)\right]\\&=\mathbb{E}\left[\exp\left(iu\Re\left(\sum_{n\leq Y}\frac{\Lambda(n)\lambda_{\pi}(n)\mathbb{X}_n}{n^{1+it}}\right)+iv\Im\left(\sum_{n\leq Y}\frac{\Lambda(n)\lambda_{\pi}(n)\mathbb{X}_n}{n^{1+it}}\right)\right)\right] +O\left(\frac{|u|+|v|}{Y^{\frac12-(\theta+\epsilon)}}\right).
\end{align*}
\end{lem}
\begin{proof}

To simplify the exposition, we demonstrate the argument for $\mathbb{E}\left[\exp\left(iu\Re\left(-\ld(1+it,\pi, \mathbb{X})\right)\right)\right]$ rather than $\mathbb{E}\left[\exp\left(iu\Re\left(-\ld(1+it,\pi, \mathbb{X})\right)+iv\Im\left(-\ld(1+it,\pi, \mathbb{X})\right)\right)\right]$, for otherwise the expressions would become quite lengthy. 
We have 
\begin{align*}
&\mathbb{E}\left[\exp\left(iu\Re\left(-\ld(1+it,\pi, \mathbb{X})\right)\right)\right]
\\&=\mathbb{E}\left[\exp\left(iu\Re\left(\sum_{p^m\leq Y}\frac{(\log p)\lambda_{\pi}(p^m)\mathbb{X}_{p^m}}{p^{m(1+it)}}+\sum_{p>Y}\frac{(\log p)\lambda_{\pi}(p)\mathbb{X}_p}{p^{1+it}}+\sum_{\substack{m\geq2\\p^m> Y}}\frac{(\log p)\lambda_{\pi}(p^m)\mathbb{X}_{p^m}}{p^{m(1+it)}}\right)\right)\right].
\end{align*}
 It follows that, for $\epsilon>0$, \begin{align*}
\sum_{\substack{m\geq2\\p^m> Y}}\frac{(\log p)\lambda_{\pi}(p^m)\mathbb{X}_{p^m}}{p^{m(1+it)}}&\ll\sum_{m\geq2}\frac{1}{m}\sum_{p>Y^{\frac{1}{m}}}\frac{1}{p^{m(1-\theta-\epsilon)}}\\&\ll\sum_{m\geq2}\frac{1}{m}\left(\frac{Y^{\frac{1}{m}-(1-\theta-\epsilon)}}{m(1-\theta-\epsilon)-1}\right)\\&\ll Y^{\theta+\epsilon-1}Y^{\frac12}\sum_{m\geq2}\frac{1}{m^2}\ll \frac{Y^{\theta+\epsilon}}{Y^{\frac12}}.
\end{align*}
 Hence,
\begin{equation}
\label{eqn:lem3.3-1}
\begin{split}
&\mathbb{E}\left[\exp\left(iu\Re\left(-\ld(1+it,\pi,\mathbb{X})\right)\right)\right]\\&=\mathbb{E}\left[\exp\left(iu\Re\left(\sum_{p^m\leq Y}\frac{(\log p)\lambda_{\pi}(p^m)\mathbb{X}_{p^m}}{p^{m(1+it)}}+\sum_{p>Y}\frac{(\log p)\lambda_{\pi}(p)\mathbb{X}_p}{p^{1+it}}\right)+O\left(\frac{|u|}{Y^{\frac12-\theta-\epsilon}}\right)\right)\right].
\end{split}
\end{equation}
Now if $|u|\leq Y^{\frac{1}{2}-\theta-\epsilon}$ and $p>Y$, then 
\begin{equation}
\label{eqn:lem3.3-2}
\begin{split}\mathbb{E}\left[\exp\left(iu\Re\left(\frac{(\log p)\lambda_{\pi}(p) \mathbb{X}_p}{p^{1+it}}\right)\right)\right]&=\mathbb{E}\left[1+iu\Re\left(\frac{(\log p)\lambda_{\pi}(p) \mathbb{X}_p}{p^{1+it}}\right)+\sum_{m=2}^{\infty}\frac{(iu)^m\Re^m\left(\frac{(\log p)\lambda_{\pi}(p) \mathbb{X}_p}{p^{1+it}}\right)}{m!}\right]\\&=1+O\left(\frac{(\log p)|u|^2}{p^{2-2\theta}}\right).
\end{split}
\end{equation}
It follows from \eqref{eqn:lem3.3-1} and \eqref{eqn:lem3.3-2}  that
\begin{equation}
\label{eqn:lem3.3-3}
\begin{split}
&\mathbb{E}\left[\exp\left(iu\Re\left(-\ld(1+it,\pi, \mathbb{X})\right)\right)\right]
\\&=\mathbb{E}\left[\exp\left(iu\Re\left(\sum_{p^m\leq Y}\frac{(\log p)\lambda_{\pi}(p^m)\mathbb{X}_{p^m}}{p^{m(1+it)}}\right)+\sum_{p>Y}\log\left(1+O\left(\frac{(\log p)|u|^2}{p^{2-2\theta}}\right)\right)+O\left(\frac{|u|}{Y^{\frac12-\theta-\epsilon}}\right)\right)\right].
\end{split}
\end{equation}
Now since 
\[\sum_{p>Y}\log\left(1+O\left(\frac{(\log p)|u|^2}{p^{2-2\theta}}\right)\right)\ll |u|^2Y^{-1+2\theta+\epsilon}\ll |u|Y^{-\frac12+\theta},\]
 by \eqref{eqn:lem3.3-3}, we get
\begin{equation}\label{eqn:real}
\begin{split}
&\mathbb{E}\left[\exp\left(iu\Re\left(-\ld(1+it,\pi, \mathbb{X})\right)\right)\right]\\&=\mathbb{E}\left[\exp\left(iu\Re\left(\sum_{n\leq Y}\frac{\Lambda(n)\lambda_{\pi}(n)\mathbb{X}_n}{n^{1+it}}\right)+O\left(\frac{|u|}{Y^{\frac12-\theta-\epsilon}}\right)\right)\right]\\&=\mathbb{E}\left[\exp\left(iu\Re\left(\sum_{n\leq Y}\frac{\Lambda(n)\lambda_{\pi}(n)\mathbb{X}_n}{n^{1+it}}\right)\right)\right]+O\left(\frac{|u|}{Y^{\frac12-\theta-\epsilon}}\right).
\end{split}
\end{equation}

One can easily check that the above argument applied to \[\mathbb{E}\left[\exp\left(iu\Re\left(-\ld(1+it,\pi, \mathbb{X})\right)+iv\Im\left(-\ld(1+it,\pi, \mathbb{X})\right)\right)\right]\] yields the desired result.
\end{proof}

\section{Bridging Lemmas}\label{bridging}
In this section, we prove two results which serve a crucial role as a bridge between the arithmetic setting and the probabilistic random setting developed in the previous sections.
\begin{lem}\label{lem3.5}
Let $A\geq1$ be fixed and set $Y=(\log N)^A$. Then for any positive integers $j,\ell$ such that $j+\ell\leq\frac{\log N}{6A\log\log N}$, we have 
\begin{align*}
&\frac{1}{\left|\mathcal{F}(N)\right|}\sum_{D\in\mathcal{F}(N)}\left(\sum_{n\leq Y}\frac{\Lambda(n)\lambda_{\pi}(n)\chi_{D}(n)}{n^{1+it}}\right)^j\left(\sum_{n\leq Y}\frac{\Lambda(n)\overline{\lambda_{\pi}(n)}\chi_{D}(n)}{n^{1-it}}\right)^{\ell}\\&=\mathbb{E}\left[\left(\sum_{n\leq Y}\frac{\Lambda(n)\lambda_{\pi}(n)\mathbb{X}_n}{n^{1+it}}\right)^j\left(\sum_{n\leq Y}\frac{\Lambda(n)\overline{\lambda_{\pi}(n)}\mathbb{X}_n}{n^{1-it}}\right)^{\ell}\right] +O\left(\frac{(Y^{\frac14}\log Y)^{j+\ell}(\log N)^{\frac12}}{N^{\frac12}}\right).
\end{align*}
\end{lem}
\begin{proof}
We have 
\begin{equation}\label{eqn:jl}
\begin{split}
&\frac{1}{\left|\mathcal{F}(N)\right|}\sum_{D\in\mathcal{F}(N)}\left(\sum_{n\leq Y}\frac{\Lambda(n)\lambda_{\pi}(n)\chi_{D}(n)}{n^{1+it}}\right)^j\left(\sum_{n\leq Y}\frac{\Lambda(n)\overline{\lambda_{\pi}(n)}\chi_{D}(n)}{n^{1-it}}\right)^{\ell}\\&=\sum_{\substack{p_1^{m_1},\cdots, p_j^{m_j}\leq Y\\q_1^{n_1},\cdots, q_\ell^{n_\ell}\leq Y}}\frac{\left(\log p_1\right)\cdots\left(\log p_{j}\right)\left(\log q_1\right)\cdots\left(\log q_\ell\right)\lambda_{\pi}(p_{1}^{m_1})\cdots\lambda_{\pi}(p_j^{m_j})\overline{\lambda_{\pi}(q_1^{n_1})}\cdots\overline{\lambda_{\pi}(q_\ell^{n_\ell})}}{p_1^{m_1(1+it)}\cdots p_j^{m_j(1+it)}q_1^{n_1(1-it)}\cdots q_\ell^{n_\ell(1-it)}}\\&\hspace{5em}\times\frac{1}{\left|\mathcal{F}(N)\right|}\sum_{D\in\mathcal{F}(N)}\chi_{D}\left(p_1^{m_1}\cdots p_{j}^{m_j}q_1^{n_1}\cdots q_{\ell}^{n_\ell}\right).
\end{split}
\end{equation}

Note that   
\begin{equation}\label{charactersum-squares}\sum_{D \in \mathcal{F}(N)} \chi_D(m^2) =  \sum_{\substack{D \in \mathcal{F}(N) \\ (D,m) = 1}} 1 = \frac{6}{\pi^2} N \prod_{p | m} \bigg( \frac{p}{p+1} \bigg) + O \big( N^{\frac12} \tau(m) \big),\end{equation}
where $\tau(m)$ is the divisor function (see for example \cite[page~1017]{GS}). In view of \eqref{eq-E(X(n))} and \eqref{charactersum-squares}, we see that \eqref{eqn:jl} equals 
\begin{align*}&=\mathbb{E}\left[\left(\sum_{n\leq Y}\frac{\Lambda(n)\lambda_{\pi}(n)\mathbb{X}_n}{n^{1+it}}\right)^j\left(\sum_{n\leq Y}\frac{\Lambda(n)\overline{\lambda_{\pi}(n)}\mathbb{X}_n}{n^{1-it}}\right)^{\ell}\right] \\&\hspace{2em}+\sum_{\substack{p_1^{m_1},\cdots, p_j^{m_j}\leq Y\\q_1^{n_1},\cdots, q_\ell^{n_\ell}\leq Y\\p_1^{m_1}\cdots p_{j}^{m_j}q_1^{n_1}\cdots q_\ell^{n_\ell}\neq\square}}\frac{\prod_{s=1}^{j}\left(\log p_s\right)\lambda_{\pi}(p_{s}^{m_s})\prod_{s=1}^{\ell}\left(\log q_s\right)\overline{\lambda_{\pi}(q_s^{n_s})}}{p_1^{m_1(1+it)}\cdots p_j^{m_j(1+it)}q_1^{n_1(1-it)}\cdots q_\ell^{n_\ell(1-it)}}\\&\hspace{8em}\times\frac{1}{\left|\mathcal{F}(N)\right|}\sum_{D\in\mathcal{F}(N)}\chi_{D}\left(p_1^{m_1}\cdots p_{j}^{m_j}q_1^{n_1}\cdots q_{\ell}^{n_\ell}\right)+O\left(N^{-\frac12}\right).
\end{align*}

Let us now analyze the contribution of non-squares. By \eqref{eqn:Polya-Vinog}, we get 
\begin{equation}\label{eqn:jl1}
\begin{split}
&\sum_{\substack{p_1^{m_1},\cdots, p_j^{m_j}\leq Y\\q_1^{n_1},\cdots, q_k^{n_\ell}\leq Y\\p_1^{m_1}\cdots p_{j}^{m_j}q_1^{n_1}\cdots q_\ell^{n_\ell}\neq\square}}\frac{\prod_{s=1}^{j}\left(\log p_s\right)\lambda_{\pi}(p_{s}^{m_s})\prod_{s=1}^{\ell}\left(\log q_s\right)\overline{\lambda_{\pi}(q_s^{n_s})}}{p_1^{m_1(1+it)}\cdots p_j^{m_j(1+it)}q_1^{n_1(1-it)}\cdots q_\ell^{n_\ell(1-it)}}\\&\hspace{5em}\times\frac{1}{\left|\mathcal{F}(N)\right|}\sum_{D\in\mathcal{F}(N)}\chi_{D}\left(p_1^{m_1}\cdots p_{j}^{m_j}q_1^{n_1}\cdots q_{\ell}^{n_\ell}\right)\\&\ll \frac{((j+\ell)\log Y)^\frac{1}{2}}{N^{\frac12}}\sum_{\substack{p_1^{m_1}\cdots p_j^{m_j}\leq Y\\q_1^{n_1}\cdots q_\ell^{n_\ell}\leq Y\\p_1^{m_1}\cdots p_{j}^{m_j}q_1^{n_1}\cdots q_\ell^{n_\ell}\neq\square}}\frac{\prod_{s=1}^{j}\left(\log p_s\right)\left|\lambda_{\pi}(p_{s}^{m_s})\right|\prod_{s=1}^{\ell}\left(\log q_s\right)\left|\lambda_{\pi}(q_s^{n_s})\right|}{p_1^{\frac34m_1}\cdots p_j^{\frac34m_j}q_1^{\frac34n_1}\cdots q_\ell^{\frac34n_\ell}}.
\end{split}
\end{equation}

Observe that \eqref{eqn:jl1} is \begin{align*}&\ll\frac{((j+\ell)\log Y)^\frac{1}{2}}{N^{\frac12}}\left(\sum_{p^m\leq Y}\frac{\log p\left|\lambda_{\pi}(p^m)\right|}{p^{\frac34m}}\right)^{j+\ell}\\&\ll\frac{((j+\ell)\log Y)^\frac{1}{2}}{N^{\frac12}}\left(\sum_{p^m\leq Y}\frac{\log p\left|\lambda_{\pi}(p^m)\right|^2}{p^{m}}\right)^{\frac{j+\ell}{2}}\left(\sum_{p^m\leq Y}\frac{\log p}{p^{\frac{m}{2}}}\right)^{\frac{j+\ell}{2}}.\end{align*}
Since $\sum_{p^m\leq Y}\frac{\log p\left|\lambda_{\pi}(p^m)\right|^2}{p^{m}}=O(\log Y)$ (see \cite[Equation~(3)]{A-S}) and $\sum_{p^m\leq Y}\frac{\log p}{p^{\frac{m}{2}}}=O(Y^{\frac12}\log Y)$, we conclude that  \eqref{eqn:jl1} is
\begin{align*}&\ll (Y^{\frac14}\log Y)^{j+\ell}\frac{((j+\ell)\log Y)^\frac{1}{2}}{N^{\frac12}}\\&\ll \frac{(Y^{\frac14}\log Y)^{j+\ell}(\log N)^{\frac12}}{N^{\frac12}}.
\end{align*}
\end{proof}
\begin{prop}\label{prop2.3}
Let $A\geq1$ and $Y=(\log N)^{A}$. Assume Hypothesis H. Then there exist positive constants $b_{0}=b_0(A)$ and $c_0=c_0(A)$ such that for all complex numbers $z_1, z_2$ with $|z_1|,|z_2|\leq c_0\frac{\log N}{(\log\log N)^2}$, we have 
\begin{align*}
&\frac{1}{\left|\mathcal{F}(N)\right|}\sum_{D\in\mathcal{F}(N)}\exp\left(z_1\sum_{n\leq Y}\frac{\Lambda(n)\lambda_{\pi}(n)\chi_{D}(n)}{n^{1+it}}+z_2\sum_{n\leq Y}\frac{\Lambda(n)\overline{\lambda_{\pi}(n)}\chi_{D}(n)}{n^{1-it}}\right)\\&=\mathbb{E}\left[\exp\left(z_{1}\sum_{n\leq Y}\frac{\Lambda(n)\lambda_{\pi}(n)\mathbb{X}_n}{n^{1+it}}+z_2\sum_{n\leq Y}\frac{\Lambda(n)\overline{\lambda_{\pi}(n)}\mathbb{X}_n}{n^{1-it}}\right)\right]+O\left(\exp\left(-b_0\frac{\log N}{\log\log N}\right)\right).
\end{align*}
\end{prop}
\begin{proof}
Let $k=\max(|z_1|, |z_2|)$ and $M=\left[\frac{\log N}{c\log\log N}\right]$, where $c$ is a suitably large positive constant. We have
\begin{equation}\label{eqn:prop4.2-1}\begin{split}
&\frac{1}{\left|\mathcal{F}(N)\right|}\sum_{D\in\mathcal{F}(N)}\exp\left(z_1\sum_{n\leq Y}\frac{\Lambda(n)\lambda_{\pi}(n)\chi_{D}(n)}{n^{1+it}}+z_2\sum_{n\leq Y}\frac{\Lambda(n)\overline{\lambda_{\pi}(n)}\chi_{D}(n)}{n^{1-it}}\right)\\&=\sum_{j+\ell\leq M}\frac{z_1^jz_2^\ell}{j!\ell!}\frac{1}{\left|\mathcal{F}(N)\right|}\sum_{D\in\mathcal{F}(N)}\left(\sum_{n\leq Y}\frac{\Lambda(n)\lambda_{\pi}(n)\chi_{D}(n)}{n^{1+it}}\right)^j\left(\sum_{n\leq Y}\frac{\Lambda(n)\overline{\lambda_{\pi}(n)}\chi_{D}(n)}{n^{1-it}}\right)^\ell +E_1,
\end{split}\end{equation}
where \begin{align*}
E_1&= \sum_{j+\ell>M}\frac{z_1^jz_2^\ell}{j!\ell!}\frac{1}{\left|\mathcal{F}(N)\right|}\sum_{D\in\mathcal{F}(N)}\left(\sum_{n\leq Y}\frac{\Lambda(n)\lambda_{\pi}(n)\chi_{D}(n)}{n^{1+it}}\right)^j\left(\sum_{n\leq Y}\frac{\Lambda(n)\overline{\lambda_{\pi}(n)}\chi_{D}(n)}{n^{1-it}}\right)^\ell.\end{align*}

Observe that for $Y=(\log N)^A$ we have
\begin{align*}
\left|\sum_{n\leq Y}\frac{\Lambda(n)\lambda_{\pi}(n)\chi_{D}(n)}{n^{1+it}}\right|&\leq \sum_{n\leq Y}\frac{\Lambda(n)\left|\lambda_{\pi}(n)\right|}{n}\\&\leq\left(\frac12 \sum_{n\leq Y}\frac{\Lambda(n)}{n}+\frac12 \sum_{n\leq Y}\frac{\Lambda(n)\left|\lambda_{\pi}(n)\right|^2}{n}\right)\\&\leq C_0\log\log N,
\end{align*}
where $C_0$ is a positive constant that depends only on A and $\pi$.
Hence,
\begin{align*}
E_1&\ll\sum_{j+\ell> M}\frac{k^{j+\ell}}{j!\ell!}\left(C_0\log\log N\right)^{j+\ell}\\&\leq\sum_{n>M}\frac{(2C_0k\log\log N)^n}{n!},\end{align*} where we used the identity $\sum_{j=0}^n\binom{n}{j}=2^n$. By Stirling's inequality $\frac{1}{n!}\leq \left(\frac{e}{n}\right)^n$ , we get
\begin{align*}E_1&\leq \sum_{n>M}\left(\frac{2C_0ek\log\log N}{n}\right)^{n}.\end{align*}
Assuming $k\leq c_0\frac{\log N}{(\log \log N)^2}$, where we choose $c_0\leq \frac{e^{-1}}{6cC_0}$, we deduce that
\begin{align}
\label{eqn:prop4.2-2}
E_1\ll\sum_{n>M}(6cc_0C_0)^n\ll e^{-M}.
\end{align}

To analyze the main term of \eqref{eqn:prop4.2-1}, note that, by Lemma \ref{lem3.5}, we have
\begin{align*}
&\frac{1}{\left|\mathcal{F}(N)\right|}\sum_{D\in\mathcal{F}(N)}\left(\sum_{n\leq Y}\frac{\Lambda(n)\lambda_{\pi}(n)\chi_{D}(n)}{n^{1+it}}\right)^j\left(\sum_{n\leq Y}\frac{\Lambda(n)\overline{\lambda_{\pi}(n)}\chi_{D}(n)}{n^{1-it}}\right)^\ell\\&= \mathbb{E}\left[\left(\sum_{n\leq Y}\frac{\Lambda(n)\lambda_{\pi}(n)\mathbb{X}_n}{n^{1+it}}\right)^j\left(\sum_{n\leq Y}\frac{\Lambda(n)\overline{\lambda_{\pi}(n)}\mathbb{X}_n}{n^{1-it}}\right)^{\ell}\right] +O\left(\frac{(Y^{\frac14}\log Y)^{j+\ell}(\log N)^{\frac12}}{N^{\frac12}}\right)
\end{align*}
for $j+\ell\leq M$. Hence, we have 
\begin{align*}
&\frac{1}{\left|\mathcal{F}(N)\right|}\sum_{D\in\mathcal{F}(N)}\exp\left(z_1\sum_{n\leq Y}\frac{\Lambda(n)\lambda_{\pi}(n)\chi_{D}(n)}{n^{1+it}}+z_2\sum_{n\leq Y}\frac{\Lambda(n)\overline{\lambda_{\pi}(n)}\chi_{D}(n)}{n^{1-it}}\right)\\&=\sum_{j+\ell\leq M}\frac{z_1^jz_2^\ell}{j!\ell!} \mathbb{E}\left[\left(\sum_{n\leq Y}\frac{\Lambda(n)\lambda_{\pi}(n)\mathbb{X}_n}{n^{1+it}}\right)^j\left(\sum_{n\leq Y}\frac{\Lambda(n)\overline{\lambda_{\pi}(n)}\mathbb{X}_n}{n^{1-it}}\right)^{\ell}\right] \\&\hspace{1em}+O\left(\sum_{j+\ell\leq M}\frac{z_1^jz_2^\ell}{j!\ell!}\frac{(Y^{\frac14}\log Y)^{j+\ell}(\log N)^{\frac12}}{N^{\frac12}}\right).
\end{align*}
Observe that 
\begin{align*}
\sum_{j+\ell\leq M}\frac{z_1^jz_2^\ell}{j!\ell!}\frac{(Y^{\frac14}\log Y)^{j+\ell}(\log N)^{\frac12}}{N^{\frac12}}&\leq\frac{(Y^{\frac14}\log Y)^{M}(\log N)^{\frac12}}{N^{\frac12}}\sum_{n\leq M}\frac{(2k)^{n}}{n!}\\&\ll Y^{-\epsilon M}\exp(2k)\\&\ll \exp\left(-\frac{\epsilon A}{c}\log N+2c_0\frac{\log N}{(\log\log N)^2}\right),
\end{align*}
provided that $N^{\frac12}\gg Y^{(1+\epsilon)M}$ which is possible by choosing $c$ suitably large. It follows that 
\begin{equation}\label{eqn:truncated-main}
\begin{split}
&\frac{1}{\left|\mathcal{F}(N)\right|}\sum_{D\in\mathcal{F}(N)}\exp\left(z_1\sum_{n\leq Y}\frac{\Lambda(n)\lambda_{\pi}(n)\chi_{D}(n)}{n^{1+it}}+z_2\sum_{n\leq Y}\frac{\Lambda(n)\overline{\lambda_{\pi}(n)}\chi_{D}(n)}{n^{1-it}}\right)\\&=\sum_{j+\ell\leq M}\frac{z_1^jz_2^\ell}{j!\ell!} \mathbb{E}\left[\left(\sum_{n\leq Y}\frac{\Lambda(n)\lambda_{\pi}(n)\mathbb{X}_n}{n^{1+it}}\right)^j\left(\sum_{n\leq Y}\frac{\Lambda(n)\overline{\lambda_{\pi}(n)}\mathbb{X}_n}{n^{1-it}}\right)^{\ell}\right]\\&\hspace{1em}+O\left(\exp\left(-\frac{\epsilon A}{c}\log N+2c_0\frac{\log N}{(\log\log N)^2}\right)\right).
\end{split}
\end{equation}

Now the main term in \eqref{eqn:truncated-main} can be written as 
\begin{equation}\label{eqn:prop4.2-3}
\begin{split}
&\sum_{j+\ell\leq M}\frac{z_1^jz_2^\ell}{j!\ell!} \mathbb{E}\left[\left(\sum_{n\leq Y}\frac{\Lambda(n)\lambda_{\pi}(n)\mathbb{X}_n}{n^{1+it}}\right)^j\left(\sum_{n\leq Y}\frac{\Lambda(n)\overline{\lambda_{\pi}(n)}\mathbb{X}_n}{n^{1-it}}\right)^{\ell}\right]\\&=
 \mathbb{E}\left[\exp\left(z_1\left(\sum_{n\leq Y}\frac{\Lambda(n)\lambda_{\pi}(n)\mathbb{X}_n}{n^{1+it}}\right)+z_2\left(\sum_{n\leq Y}\frac{\Lambda(n)\overline{\lambda_{\pi}(n)}\mathbb{X}_n}{n^{1-it}}\right)\right)\right] +E_2,
 \end{split}
 \end{equation}
 where \begin{align*}E_2&=-\sum_{j+\ell> M}\frac{z_1^jz_2^\ell}{j!\ell!} \mathbb{E}\left[\left(\sum_{n\leq Y}\frac{\Lambda(n)\lambda_{\pi}(n)\mathbb{X}_n}{n^{1+it}}\right)^j\left(\sum_{n\leq Y}\frac{\Lambda(n)\overline{\lambda_{\pi}(n)}\mathbb{X}_n}{n^{1-it}}\right)^{\ell}\right].\end{align*}
 Observe that \begin{align*}E_2&\ll\sum_{j+\ell> M}\frac{k^{j+\ell}}{j!\ell!} \mathbb{E}\left[\left|\sum_{n\leq Y}\frac{\Lambda(n)\lambda_{\pi}(n)\mathbb{X}_n}{n^{1+it}}\right|^{j+\ell}\right]\\&\ll\sum_{j+\ell> M}\frac{k^{j+\ell}}{j!\ell!} \left(\mathbb{E}\left[\left|\sum_{n\leq Y}\frac{\Lambda(n)\lambda_{\pi}(n)\mathbb{X}_n}{n^{1+it}}\right|^{2(j+\ell)}\right]\right)^{\frac12}.
 \end{align*}
Using Lemma \ref{lem3.3-random} we get \[E_2\ll \sum_{m> M}\frac{(2k)^{m}}{m!} \left(C\log m\right)^{m}.\] 
By Stirling's formula and our choices for $k$ and $M$, we deduce that  \begin{equation}\label{eqn:prop4.2-4}E_2\ll \sum_{m>M}\left(\frac{6Ck\log m}{m}\right)^m\ll\sum_{m>M}\left(\frac{6Cc_0\log N\log M}{M(\log\log N)^2}\right)^m\ll \sum_{m>M}(6Ccc_0)^m\ll e^{-M},\end{equation} where the last inequality follows from choosing $c_0\leq \frac{e^{-1}}{6cC}$. 

From \eqref{eqn:prop4.2-1}, \eqref{eqn:prop4.2-2}, \eqref{eqn:truncated-main}, \eqref{eqn:prop4.2-3}, and \eqref{eqn:prop4.2-4} we obtain the desired result.
 \end{proof}

\section{The Characteristic Function}\label{char-function}

We set \[\Phi_{\mathcal{F}(N)}(u,v):=\frac{1}{\left|\mathcal{F}(N)\right|}\sum_{D\in\mathcal{F}(N)}\exp\left(iu\Re\left(-\frac{L'}{L}(1+it,\pi\otimes\chi_{D})\right)+iv\Im\left(-\frac{L'}{L}(1+it,\pi\otimes\chi_{D})\right)\right)\] and
\[\Phi_{\mathcal{A}(N)}(u,v):=\frac{1}{\left|\mathcal{A}(N)\right|}\sum_{D\in\mathcal{A}(N)}\exp\left(iu\Re\left(-\frac{L'}{L}(1+it,\pi\otimes\chi_{D})\right)+iv\Im\left(-\frac{L'}{L}(1+it,\pi\otimes\chi_{D})\right)\right),\] where $\mathcal{A}(N)$ is the set of fundamental discriminants introduced after Proposition \ref{prop:short}.
We also define \[\Phi_{\mathrm{rand}}(u,v):=\mathbb{E}\left[\exp\left(iu\Re\left(-\ld(1+it,\pi, \mathbb{X})\right)+iv\Im\left(-\ld(1+it,\pi, \mathbb{X})\right)\right)\right].\]

\begin{thrm}
\label{FourierTransform}
Let $A\geq 1$ be fixed. Under the assumption of the Hypothesis H, there exists a positive constant $c_0=c_0(A)$  such that for $|u|,|v|\leq c_0\frac{\log N}{(\log\log N)^2}$, we have 
\begin{align}\label{eqn:arith-prob}
\Phi_{\mathcal{F}(N)}(u,v)
&=\Phi_{\mathrm{rand}}(u,v)
+  O\left(\frac{1}{(\log\log{N})^2(\log N)^{A}}\right).
\end{align}
Moreover, the asymptotic formula \eqref{eqn:arith-prob} holds if we replace $\Phi_{\mathcal{F}(N)}(u,v)$ by $\Phi_{\mathcal{A}(N)}(u,v)$.
\end{thrm}
\begin{proof}
From Proposition \ref{prop:short} we know that for all but $O(N^{\frac34})$ elements $D\in\mathcal{F}(N)$, there are positive constants $\delta_\pi$ and  $\eta_\pi$ such that for $(\log N)^{\eta_\pi}\ll Y\ll N^{3d^2}$, we have
\begin{equation}\label{eqn:dirichlet-poly}
-\frac{L'}{L}(1+it,\pi\otimes\chi_{D})=\sum_{n\leq Y}\frac{\Lambda(n)\lambda_{\pi}(n)\chi_D(n)}{n^{1+it}} +O(Y^{-\delta_\pi}).
\end{equation}
 Recall that $\mathcal{A}(N)$ is the set of elements for which \eqref{eqn:dirichlet-poly} holds. Let $0\leq\theta<\frac12$ be such that $|\lambda_{\pi}(p)|\ll p^{\theta}$ (such $\theta$ exists because of \eqref{RS-bound}). Let $0<\epsilon<\frac{1}{2}-\theta$, and choose  $0<\delta<\min(\frac12-\theta-\epsilon,\delta_{\pi})$ such that $\frac{A+1}{\delta}\geq \eta_\pi$, and set $Y=(\log N)^{\frac{A+1}{\delta}}$. Applying \eqref{eqn:dirichlet-poly} and the inequality $|e^{ia}-e^{ib}|\leq |b-a|$, we get
\begin{equation}\label{eqn:thm5.1-1}\begin{split}
\Phi_{\mathcal{F}(N)}(u,v)
&=\frac{1}{\left|\mathcal{F}(N)\right|}\sum_{D\in\mathcal{A}(N)}\exp\left(iu\Re\left(-\frac{L'}{L}(1+it,\pi\otimes\chi_{D})\right)+iv\Im\left(-\frac{L'}{L}(1+it,\pi\otimes\chi_{D})\right)\right)\\&\hspace{2em}+O(N^{-\frac14})\\&=\frac{1}{\left|\mathcal{F}(N)\right|}\sum_{D\in\mathcal{A}(N)}\exp\left(iu\Re\left(\sum_{n\leq Y}\frac{\Lambda(n)\lambda_{\pi}(n)\chi_D(n)}{n^{1+it}}\right)+iv\Im\left(\sum_{n\leq Y}\frac{\Lambda(n)\lambda_{\pi}(n)\chi_D(n)}{n^{1+it}}\right)\right)\\&\hspace{2em}+O\left((|u|+|v|)Y^{-\delta}\right)+O(N^{-\frac14}).
\end{split}\end{equation}
Notice that \[\exp\left(O\left((|u|+|v|)\right)Y^{-\delta}\right)
=1+O\left((|u|+|v|)(\log N)^{-(A+1)}\right)\] whenever $|u|+|v|<(\log N)^{A+1}$. If we assume further that $|u|+|v|\leq c_0\frac{\log N}{(\log\log N)^2}$ for some $c_0>0$, from \eqref{eqn:thm5.1-1}, we get 
\begin{align*}
\Phi_{\mathcal{F}(N)}(u,v)&=
\frac{1}{\left|\mathcal{F}(N)\right|}\sum_{D\in\mathcal{F}(N)}\exp\left(iu\Re\left(\sum_{n\leq Y}\frac{\Lambda(n)\lambda_{\pi}(n)\chi_D(n)}{n^{1+it}}\right)+iv\Im\left(\sum_{n\leq Y}\frac{\Lambda(n)\lambda_{\pi}(n)\chi_D(n)}{n^{1+it}}\right)\right)\\&\hspace{2em}+O\left(\frac{1}{(\log\log N)^2(\log N)^{A}}\right).
\end{align*}
Here we used the fact that $\left|\mathcal{F}(N)\setminus\mathcal{A}(N)\right|=O(N^{\frac34})$ by Proposition \ref{prop:short}. 

Next, since $|u|,|v|\leq c_0\frac{\log N}{(\log\log N)^2}$, Proposition \ref{prop2.3} with $z_1=\frac{i}{2}(u+iv)$ and $z_2=\frac{i}{2}(u-iv)$ gives
\begin{equation}\label{eqn:phi-1}
\begin{split}
\Phi_{\mathcal{F}(N)}(u,v)&= \mathbb{E}\left[\exp\left(z_{1}\sum_{n\leq Y}\frac{\Lambda(n)\lambda_{\pi}(n)\mathbb{X}_n}{n^{1+it}}+z_2\sum_{n\leq Y}\frac{\Lambda(n)\overline{\lambda_{\pi}(n)}\mathbb{X}_n}{n^{1-it}}\right)\right]\\&\hspace{2em}+O\left(\frac{1}{(\log\log N)^2(\log N)^{A}}\right)+O\left(\exp\left(-b_0\frac{\log N}{\log\log N}\right)\right)
\end{split}
\end{equation}
for some $b_0>0$. 
Now employing Lemma \ref{lem4.1} in \eqref{eqn:phi-1} yields 
\begin{align*}
\Phi_{\mathcal{F}(N)}(u,v)&=\mathbb{E}\left[\exp\left(iu\Re\left(-\ld(1+it,\pi,\mathbb{X})\right)+iv\Im\left(-\ld(1+it,\pi, \mathbb{X})\right)\right)\right]\\&\hspace{2em}+ O\left(\frac{1}{(\log\log N)^2(\log N)^{\frac{A+1}{\delta}(\frac12-\theta-\epsilon)-1}}\right)+ O\left(\frac{1}{(\log\log N)^2(\log N)^{A}}\right).
\end{align*}
The desired error term is achieved since we assume that $\delta<\frac12-\theta-\epsilon$. 

The same asymptotic formula holds for $\Phi_{\mathcal{A}(N)}(u,v)$ since $\Phi_{\mathcal{F}(N)}(u,v)=\Phi_{\mathcal{A}(N)}(u,v)+O(N^{-\frac14}).$
\end{proof}

\section{The Exponential Decay}\label{decay}

In this section we prove a decay bound for $\Phi_{\mathrm{rand}}(u,v)$. Recall that we are working under the assumption that the distributions are 2-dimensional, hence, $\pi\not\cong\tilde{\pi}\otimes\left|\mathrm{det}\right|^{2it}$. 
The following proposition implies among other things that the distribution function associated with our random series admits a smooth density function. 

\begin{prop}\label{exponential-decay} 
Let $z=u+iv$ and assume that $\left|\lambda_{\pi}(p)\right|\ll p^{\theta}$ for some $0\leq\theta<\frac14$. Let $0<\epsilon<2-8\theta$ be given. Then there exist positive constants $C_{\epsilon}$ and $a_{\epsilon}$ depending only on $\epsilon$ such that for $|z|$ large enough we have
\[\Phi_{\mathrm{rand}}(u,v)\ll \exp\left(-C_{\epsilon}|z|^{a_{\epsilon}}\right).\]

\end{prop}
\begin{proof}

By employing \eqref{eqn:random-model}, we get that 
\[\Phi_{\mathrm{rand}}(u,v)=\prod_{p}M_{p}(u,v),\] where

\begin{align*}
M_{p}(u,v)&=\frac{1}{p+1}+\frac{p}{2(p+1)}\exp\left(i\left(\log p\right)\Re\left(\sum_{j=1}^d\frac{\overline{z}\alpha_{j,\pi}(p)}{p^{1+it}-\alpha_{j,\pi}(p)}\right)\right)\\&\hspace{2em}+\frac{p}{2(p+1)}\exp\left(-i\left(\log p\right)\Re\left(\sum_{j=1}^d\frac{\overline{z}\alpha_{j,\pi}(p)}{p^{1+it}+\alpha_{j,\pi}(p)}\right)\right).\end{align*}
We observe that $\left| \prod_{p}M_{p}(u,v)\right|\leq \prod_{p>X}\left|M_p(u,v)\right|$ since $|M_{p}(u,v)|\leq1$ for all $p$. Moreover,
\begin{equation*}\label{eqn:approx}
\frac{\alpha_{j,\pi}(p)}{p^{1+it}-\alpha_{j,\pi}(p)}-\frac{\alpha_{j,\pi}(p)}{p^{1+it}}=\frac{\alpha^2_{j,\pi}(p)}{p^{1+it}(p^{1+it}-\alpha_{j,\pi}(p))}=O\left(\frac{|\alpha_{j,\pi}(p)|^2}{p^2}\right).
\end{equation*}
Hence,
\begin{align*}
M_{p}(u,v)&=\frac{1}{p+1}+\frac{p}{2(p+1)}\exp\left(i\left(\log p\right)\Re\left(\sum_{j=1}^d\frac{\overline{z}\alpha_{j,\pi}(p)}{p^{1+it}}\right)\right)\\&\hspace{2em}+\frac{p}{2(p+1)}\exp\left(-i\left(\log p\right)\Re\left(\sum_{j=1}^d\frac{\overline{z}\alpha_{j,\pi}(p)}{p^{1+it}}\right)\right) +O\left(\frac{(\log p)|z|}{p^2}\sum_{j=1}^d|\alpha_{j,\pi}(p)|^2\right).
\end{align*}

Next by using the Taylor expansion of the exponential function and simplifying the resulting expressions,  from the above we get
\begin{align*}
M_{p}(u,v)&=1-\frac{p}{2(p+1)}(\log p)^2\Re^2\left(\sum_{j=1}^d\frac{\overline{z}\alpha_{j,\pi}(p)}{p^{1+it}}\right)\\&\hspace{2em}+O\left(\frac{|z|^4(\log p)^4}{p^4}\sum_{j=1}^d|\alpha_{j,\pi}(p)|^4+\frac{(\log p)|z|}{p^2}\sum_{j=1}^d|\alpha_{j,\pi}(p)|^2\right),\end{align*}
provided that $\left(\log p\right)\left|\Re\left(\sum_{j=1}^d\frac{\overline{z}\alpha_{j,\pi}(p)}{p^{1+it}}\right)\right|<1$. Hence, for sufficiently large $X$, we have 
\begin{align*}
\prod_{p>X}M_p(u,v)&=\prod_{p>X}\exp\Bigg(\log\Bigg(1-\frac{p}{2(p+1)}(\log p)^2\Re^2\left(\sum_{j=1}^d\frac{\overline{z}\alpha_{j,\pi}(p)}{p^{1+it}}\right)\\&\hspace{2em}+O\left(\frac{|z|^4(\log p)^4}{p^4}\sum_{j=1}^d|\alpha_{j,\pi}(p)|^4+\frac{(\log p)|z|}{p^2}\sum_{j=1}^d|\alpha_{j,\pi}(p)|^2\right)\Bigg)\Bigg)
\\&\leq\exp\Bigg(-\sum_{p>X}\frac{p}{2(p+1)}(\log p)^2\Re^2\left(\frac{\overline{z}a_{\pi}(p)}{p^{1+it}}\right)\\&\hspace{2em}+O\left(|z|^4\sum_{p>X}\frac{(\log p)^4}{p^4}\sum_{j=1}^d|\alpha_{j,\pi}(p)|^4+|z|\sum_{p>X}\frac{\log p}{p^2}\sum_{j=1}^d|\alpha_{j,\pi}(p)|^2\right)\Bigg),
\end{align*}
where in the above estimations we  used $\log(1-x)\leq -x$ for all $0<x<1$. In order to estimate this expression, we need to first consider the sum $\displaystyle{\sum_{p>X}\frac{p}{2(p+1)}(\log p)^2\Re^2\left(\frac{\overline{z}a_{\pi}(p)}{p^{1+it}}\right)}$. To simplify exposition, we will consider instead the sum $\displaystyle{\sum_{p>X}\frac{(\log p)^2}{p^2}\Re^2\left(\frac{\overline{z}a_{\pi}(p)}{p^{it}}\right)}$. Observe that
\begin{align*}\sum_{p>X}\frac{(\log p)^2}{p^2}\Re^2\left(\frac{\overline{z}a_{\pi}(p)}{p^{it}}\right)&=\frac12|z|^2\sum_{p>X}\frac{(\log p)^2}{p^2}|a_\pi(p)|^2\\&\hspace{2em}+\frac{1}{4}\overline{z}^2\sum_{p>X}\frac{(\log p)^2}{p^{2+2it}}a_{\pi}(p)^2+\frac14z^2\sum_{p>X}\frac{(\log p)^2}{p^{2-2it}}\overline{a_\pi(p)}^2. \end{align*}
Note that by Lemma \ref{lem:pnt1H} we have \[\frac12|z|^2\sum_{p>X}\frac{(\log p)^2}{p^2}|a_\pi(p)|^2=\frac12|z|^2\left(\frac{\log X}{X}\right)+O\left(\frac{|z|^2}{X}\right).\] If $\pi\not\cong\tilde{\pi}\otimes\left|\mathrm{det}\right|^{i\tau}$ for any $\tau\in\mathbb{R}$, then  Lemma \ref{lem:pnt2H} gives \[\frac{1}{4}\overline{z}^2\sum_{p>X}\frac{(\log p)^2}{p^{2+2it}}a_{\pi}(p)^2+\frac14z^2\sum_{p>X}\frac{(\log p)^2}{p^{2-2it}}\overline{a_\pi(p)}^2\ll |z|^2\frac{(\log X)^{\alpha}}{X},\] for some $0<\alpha<1$, and so
\[\sum_{p>X}\frac{(\log p)^2}{p^2}\Re^2\left(\frac{\overline{z}a_{\pi}(p)}{p^{it}}\right)\asymp|z|^2\frac{\log X}{X}.\]
If $\pi\cong\tilde{\pi}\otimes\left|\mathrm{det}\right|^{i\tau_0}$ for some $\tau_0\in\mathbb{R}$, then Lemma \ref{lem:pnt2H} gives \begin{align*}\frac{1}{4}\overline{z}^2\sum_{p>X}\frac{(\log p)^2}{p^{2+2it}}a_{\pi}(p)^2+\frac14z^2\sum_{p>X}\frac{(\log p)^2}{p^{2-2it}}\overline{a_\pi(p)}^2&= \frac{1}{4}\overline{z}^2\left(\frac{\log X}{(1-2it+i\tau_0)X^{1-2it+i\tau_0}}\right)\\&\hspace{2em}+\frac{1}{4}z^2\left(\frac{\log X}{(1+2it-i\tau_0)X^{1+2it-i\tau_0}}\right)+O\left(\frac{|z|^2}{X}\right).\end{align*} 
In this case, we get 
\begin{align*}
\sum_{p>X}\frac{(\log p)^2}{p^2}\Re^2\left(\frac{\overline{z}a_{\pi}(p)}{p^{it}}\right)&=\frac{1}{2}|z|^2\frac{\log X}{X}\left(1+\Re\left(\frac{\overline{z}}{z}\cdot\frac{X^{2it-i\tau_0}}{1-2it+i\tau_0}\right)\right) +O\left(\frac{|z|^2}{X}\right).
\end{align*}
Note that $1+\Re\left(\frac{\overline{z}}{z}\cdot\frac{X^{2it-i\tau_0}}{1-2it+i\tau_0}\right)\asymp 1$ since $\tau_0\neq 2t$ (notice that if $\tau_0=2t$, then the distribution is 1-dimensional). Hence, in all cases, we have \begin{equation}\label{eqn:decay1}\begin{split}\sum_{p>X}\frac{(\log p)^2}{p^2}\Re^2\left(\frac{\overline{z}a_{\pi}(p)}{p^{it}}\right)& \asymp|z|^2\frac{\log X}{X}. \end{split}\end{equation}
Next, we note that
\begin{equation}\label{eqn:decay2}
\begin{split}
|z|^4\sum_{p>X}\frac{(\log p)^4}{p^4}\sum_{j=1}^d|\alpha_{j,\pi}(p)|^4&\ll |z|^4\sum_{p>X}\frac{(\log p)^4}{p^{4-4\theta}}\ll |z|^4 \frac{(\log X)^3}{X^{3-4\theta}}
\end{split}
\end{equation}
and 
\begin{equation}\label{eqn:decay3}
\begin{split}
|z|\sum_{p>X}\frac{\log p}{p^2}\sum_{j=1}^d|\alpha_{j,\pi}(p)|^2\ll |z|\sum_{p>X}\frac{(\log p)^4}{p^{2-2\theta}}\ll \frac{|z|}{X^{1-2\theta}}.
\end{split}
\end{equation}
Thus, choosing $X=|z|^{\frac{2+2\epsilon}{2-4\theta}}$  in \eqref{eqn:decay1}, \eqref{eqn:decay2} an \eqref{eqn:decay3} guarantees that

\[\prod_{p}M_{p}(u,v)\leq\prod_{p>X}M_p(u,v)\ll \exp(-C_{\epsilon}|z|^{\frac{2-8\theta-2\epsilon}{2-4\theta}}),\] for some positive constant $C_{\epsilon}$.  Since $0\leq \theta <\frac 1 4$, we have the desired result.
\end{proof}

\section{Proof of Theorem \ref{main}}\label{proof-main}

In this section we prove Theorem \ref{main} in the case when the distributions are 2-dimensional. Our proof of Theorem \ref{main} follows from Theorem \ref{FourierTransform} by an application of a 2-dimensional version of Berry-Esseen inequality  which we state in Proposition \ref{barry-esseen}. 
The proof of Theorem \ref{main} in the 1-dimensional case follows the same argument and uses the more common 1-dimensional version of Berry-Esseen inequality (see for example \cite{loeve}). We denote by $\mathcal{B}(S)$ the collection of the Borel sets of a topological space $S$.

\begin{prop}\label{barry-esseen}\cite[Theorem 1]{sadikova}
Let $\mu$ and $\nu$ be probability measures on $(\mathbb{R}^2,\mathcal{B}(\mathbb{R}^2))$ with distribution functions
\[F(x,y)=\mu\left((-\infty,x]\times(-\infty,y]\right)\quad\quad\text{and}\quad\quad G(x,y)=\nu\left((-\infty,x]\times(-\infty,y]\right).\]
Suppose that $G$ is partially differentiable, and put \[A_1=\sup_{(x,y)\in\mathbb{R}^2}G_{x}(x,y)\quad\text{ and}\quad A_2=\sup_{(x,y)\in\mathbb{R}^2}G_{y}(x,y).\] Denote by $f$ and $g$ the characteristic functions associated with $\mu$ and $\nu$ respectively. Let $\widehat{f}(u,v)=f(u,v)-f(u,0)f(0,v)$ and $\widehat{g}(u,v)=g(u,v)-g(u,0)g(0,v)$. Then we have
\begin{align*}\sup_{(x,y)\in\mathbb{R}^2}|F(x,y)-G(x,y)|&\leq \frac{2}{(2\pi)^2}\int_{-R}^R\int_{-R}^R\left|\frac{\widehat{f}(u,v)-\widehat{g}(u,v)}{uv}\right|\;du\;dv\\&\hspace{2em}+ \frac{2}{\pi}\int_{-R}^R\left|\frac{f(u,0)-{g}(u,0)}{u}\right|\;du+\frac{2}{\pi}\int_{-R}^R\left|\frac{f(0,v)-{g}(0,v)}{v}\right|\;dv\\&\hspace{2em}+\left(3\sqrt{2}+4\sqrt{3}+\frac{24}{\pi}\right)\frac{2(A_1+A_2)}{R}\end{align*}
for any $R>0$.
\end{prop}
Recall that $\mathcal{A}(N)$ is the subset of $\mathcal{F}(N)$ for which \eqref{eqn:short} holds.  We define the probability space  $(\mu_{\mathcal{A}(N)},\mathbb{C},\mathcal{B}(\mathbb{C}))$ where the probability measure $\mu_{\mathcal{A}(N)}(A)$ is given by \[\mu_{\mathcal{A}(N)}(A)=\frac{1}{\left|\mathcal{A}(N)\right|}\sum_{D\in\mathcal{A}(N)}\mathbf{1}_{A}\left(-\frac{L'}{L}(1+it,\pi\otimes\chi_{D})\right).\]  We also define a probability space 
$(\mu_{\mathrm{rand}} ,\mathbb{C},\mathcal{B}(\mathbb{C}))$, where the probability measure $\mu_{\mathrm{rand}}$ is given by 
\[\mu_{\mathrm{rand}}(A)=\mathbb{P}\left(-\ld(1+it,\pi, \mathbb{X})\in A\right).\] 
Let $F_{\mathcal{A}(N)}(x,y)$ and $G_{\mathrm{rand}}(x,y)$ be the distribution functions associated with $\mu_{\mathcal{A}(N)}$ and $\mu_{\mathrm{rand}}$ respectively. Their characteristic functions are given by 
\[\Phi_{\mathcal{A}(N)}(u,v)=\frac{1}{|\mathcal{A}(N)|}\sum_{D\in\mathcal{A}(N)}\exp\left(iu\Re\left(-\frac{L'}{L}(1+it,\pi\otimes\chi_{D})\right)+iv\Im\left(-\frac{L'}{L}(1+it,\pi\otimes\chi_{D})\right)\right)\] and \[\Phi_{\mathrm{rand}}(u,v)=\mathbb{E}\left[\exp\left(iu\Re\left(-\ld(1+it,\pi,\mathbb{X})\right)+iv\Im\left(-\ld(1+it,\pi, \mathbb{X})\right)\right)\right].\]
It follows from Proposition \ref{exponential-decay} that the distribution function $G_{\mathrm{rand}}(x,y)$ admits a smooth density function $M_{\mathrm{rand}}(u,x)$ (see \cite[Theorem~2.1]{AH}) such that \[G_{\mathrm{rand}}(x,y)=\int_{-\infty}^x\int_{-\infty}^yM_{\mathrm{rand}}(u,v)\;du\;dv.\] Hence, $G_{\mathrm{rand}}(x,y)$ is partially differentiable. Moreover, $\displaystyle{A_1=\sup_{(x,y)\in\mathbb{R}^2}G_{x}(x,y)}$ and $\displaystyle{A_2=\sup_{(x,y)\in\mathbb{R}^2}G_{y}(x,y)}$ are finite in view of the identity \[\int\int_{\mathbb{R}^2}M_{\mathrm{rand}}(u,v)\;du\;dv=1.\] 

Now we have all the tools to prove our main theorem.

\begin{proof}[Proof of Theoem \ref{main}] Applying Proposition \ref{barry-esseen} with $\mu=\mu_{\mathcal{A}(N)}$ and $\nu=\mu_{\mathrm{rand}}$ by identifying $\mathbb{C}$ with $\mathbb{R}^2$, we get
\begin{equation}\label{eqn:barry-esseen}
\begin{split}
\sup_{(x,y)\in\mathbb{R}^2}\left|F_{\mathcal{A}(N)}(x,y)-G_{\mathrm{rand}}(x,y)\right|&\ll \int_{-R}^R\int_{-R}^R\left|\frac{\widehat{\Phi_{\mathcal{A}(N)}}(u,v)-\widehat{\Phi_{\mathrm{rand}}}(u,v)}{uv}\right|\;du\;dv\\&\hspace{2em}+ \int_{-R}^R\left|\frac{\Phi_{\mathcal{A}(N)}(u,0)-{\Phi_{\mathrm{rand}}}(u,0)}{u}\right|\;du\\&\hspace{2em}+\int_{-R}^R\left|\frac{\Phi_{\mathcal{A}(N)}(0,v)-\Phi_{\mathrm{rand}}(0,v)}{v}\right|\;dv+\frac{1}{R}
\end{split}
\end{equation}
for any $R>0$. We take $R=c_0\frac{\log N}{(\log\log N)^2}$ as in Theorem \ref{FourierTransform}. We have \begin{equation}\label{eqn:I1def}\begin{split}I_1&=\int_{-R}^R\int_{-R}^R\left|\frac{\widehat{\Phi_{\mathcal{A}(N)}}(u,v)-\widehat{\Phi_{\mathrm{rand}}}(u,v)}{uv}\right|\;du\;dv\\&=\int\int_{[-R,R]^2\setminus C(r)}\left|\frac{\widehat{\Phi_{\mathcal{A}(N)}}(u,v)-\widehat{\Phi_{\mathrm{rand}}}(u,v)}{uv}\right|\;du\;dv\\&\hspace{1em}+ \int\int_{C(r)}\left|\frac{\widehat{\Phi_{\mathcal{A}(N)}}(u,v)-\widehat{\Phi_{\mathrm{rand}}}(u,v)}{uv}\right|\;du\;dv,
\end{split}\end{equation}
where we take $r=\frac{1}{(\log N)^B}$ for some $B>1$ and $C(r)=\{(u,v)\in[-R,R]^2:|u|\leq r\;\text{or}\; |v|\leq r\}$. 

We set
\begin{align*}I_{11}&=\int\int_{[-R,R]^2\setminus C(r)}\left|\frac{\widehat{\Phi_{\mathcal{A}(N)}}(u,v)-\widehat{\Phi_{\mathrm{rand}}}(u,v)}{uv}\right|\;du\;dv.\end{align*}
Observe that \begin{align*}I_{11}&\ll \left(\log\left(\frac{R}{r}\right)\right)^2\sup_{(u,v)\in[-R,R]^2}\left|\widehat{\Phi_{\mathcal{A}(N)}}(u,v)-\widehat{\Phi_{\mathrm{rand}}}(u,v)\right|.\end{align*}
Moreover, we have
\begin{align*}&\left|\widehat{\Phi_{\mathcal{A}(N)}}(u,v)-\widehat{\Phi_{\mathrm{rand}}}(u,v)\right|\\&=\left|\Phi_{\mathcal{A}(N)}(u,v)-\Phi_{\mathcal{A}(N)}(u,0)\Phi_{\mathcal{A}(N)}(0,v)-\Phi_{\mathrm{rand}}(u,v)+\Phi_{\mathrm{rand}}(u,0)\Phi_{\mathrm{rand}}(0,v)\right|\\&\leq\left|\Phi_{\mathcal{A}(N)}(u,v)-\Phi_{\mathrm{rand}}(u,v)\right|+\left|\Phi_{\mathcal{A}(N)}(u,0)\Phi_{\mathcal{A}(N)}(0,v)-\Phi_{\mathrm{rand}}(u,0)\Phi_{\mathrm{rand}}(0,v)\right|\\&\leq\left|\Phi_{\mathcal{A}(N)}(u,v)-\Phi_{\mathrm{rand}}(u,v)\right|+\left|\Phi_{\mathcal{A}(N)}(u,0)-\Phi_{\mathrm{rand}}(u,0)\right|+\left|\Phi_{\mathcal{A}(N)}(0,v)-\Phi_{\mathrm{rand}}(0,v)\right|\\&\ll \frac{1}{(\log N)^A}, \end{align*}
where the last estimate follows from Theorem \ref{FourierTransform}. Hence, 
\begin{equation}\label{eqn:I11}I_{11}\ll\frac{1}{(\log N)^A}\left(\log\left(\frac{c_0\frac{\log N}{(\log\log N)^2}}{(\log N)^B}\right)\right)^2\ll (\log N)^{-A}(\log\log N)^2.\end{equation}
Next we set
\begin{align*}
I_{12}&=\int\int_{C(r)}\left|\frac{\widehat{\Phi_{\mathcal{A}(N)}}(u,v)-\widehat{\Phi_{\mathrm{rand}}}(u,v)}{uv}\right|\;du\;dv
\end{align*}
in \eqref{eqn:I1def}. We have 
\begin{align*}
\widehat{\Phi_{\mathcal{A}(N)}}(u,v)&=(\Phi_{\mathcal{A}(N)}(u,v)-\Phi_{\mathcal{A}(N)}(u,0)-\Phi_{\mathcal{A}(N)}(0,v)+1)-(\Phi_{\mathcal{A}(N)}(u,0)-1)(\Phi_{\mathcal{A}(N)}(0,v)-1)\\&=\int\int_{\mathbb{R}^2}(e^{ixu}-1)(e^{iyv}-1)\;d\mu_{\mathcal{A}(N)}(x,y)\\&\hspace{2em}-\left(\int\int_{\mathbb{R}^2}(e^{ixu}-1)\;d\mu_{\mathcal{A}(N)}(x,y)\right)\left(\int\int_{\mathbb{R}^2}(e^{iyv}-1)\;d\mu_{\mathcal{A}(N)}(x,y)\right).
\end{align*}
Notice that $e^{i\theta}-1\ll|\theta|$ for any $\theta\in\mathbb{R}$. Thus,
\begin{align*}
\widehat{\Phi_{\mathcal{A}(N)}}(u,v)&\ll |uv|\int\int_{\mathbb{R}^2}|xy|\;d\mu_{\mathcal{A}(N)}(x,y)+|uv|\left(\int\int_{\mathbb{R}^2}|x|\;d\mu_{\mathcal{A}(N)}(x,y)\right)\left(\int\int_{\mathbb{R}^2}|y|\;d\mu_{\mathcal{A}(N)}(x,y)\right)\\&\ll |uv|\int\int_{\mathbb{R}^2}(x^2+y^2)\;d\mu_{\mathcal{A}(N)}(x,y)+|uv|\left(\frac{1}{|\mathcal{A}(N)|}\sum_{D\in\mathcal{A}(N)}\left|\frac{L'}{L}(1+it,\pi\otimes \chi_D)\right|^2\right)\\&\ll |uv|,
\end{align*}
where the last estimate follows from Lemma \ref{2ndmoment}. Similarly, we have $\Phi_{\mathrm{rand}}(u,v)\ll |uv|$. It follows that \begin{equation}\label{eqn:I12}I_{12}\ll \mathrm{meas}(C(r))\ll rR=(\log N)^{-B}\frac{\log N}{(\log\log N)^2}\ll (\log N)^{-B+1}.\end{equation} Hence, by \eqref{eqn:I11} and \eqref{eqn:I12}, we have 
\begin{equation}\label{eqn:I1}I_1\ll\max((\log N)^{-B+1},(\log N)^{-A}(\log\log N)^2).\end{equation}

Let us now estimate 
\begin{align*}
I_2&=\int_{-R}^R\left|\frac{\Phi_{\mathcal{A}(N)}(u,0)-{\Phi_{\mathrm{rand}}}(u,0)}{u}\right|\;du.
\end{align*}
We have 
\begin{align*}
I_2&=\int_{-R}^{-r}\left|\frac{\Phi_{\mathcal{A}(N)}(u,0)-{\Phi_{\mathrm{rand}}}(u,0)}{u}\right|\;du +\int_{r}^R\left|\frac{\Phi_{\mathcal{A}(N)}(u,0)-{\Phi_{\mathrm{rand}}}(u,0)}{u}\right|\;du\\&\hspace{2em}+\int_{-r}^r\left|\frac{\Phi_{\mathcal{A}(N)}(u,0)-{\Phi_{\mathrm{rand}}}(u,0)}{u}\right|\;du.
\end{align*}
Observe that 
\begin{align*}
&\int_{-R}^{-r}\left|\frac{\Phi_{\mathcal{A}(N)}(u,0)-{\Phi_{\mathrm{rand}}}(u,0)}{u}\right|\;du +\int_{r}^R\left|\frac{\Phi_{\mathcal{A}(N)}(u,0)-{\Phi_{\mathrm{rand}}}(u,0)}{u}\right|\;du\\&\ll \log\left(\frac{R}{r}\right)\sup_{u\in[-R,R]}\left|\Phi_{\mathcal{A}(N)}(u,0)-\Phi_{\mathrm{rand}}(u,0)\right|\\&\ll(\log N)^{-A}\log\log N.
\end{align*}
Also notice that 
\begin{align*}
\Phi_{\mathcal{A}(N)}(u,0)-{\Phi_{\mathrm{rand}}}(u,0)&=\int\int_{\mathbb{R}^2}(e^{ixu}-1)\;d\mu_{\mathcal{A}(N)}(x,y)-\int\int_{\mathbb{R}^2}(e^{ixu}-1)\;d\mu_{\mathrm{rand}}(x,y)\\&\ll |u|\left(\int\int_{\mathbb{R}^2}x^2\;d\mu_{\mathcal{A}(N)}(x,y)\right)^{\frac12}+|u|\left(\int\int_{\mathbb{R}^2}x^2\;d\mu_{\mathrm{rand}}(x,y)\right)^{\frac12}\\&\ll |u|.
\end{align*}
For the last estimate, we used the bound 
\[\frac{1}{|\mathcal{A}(N)|}\sum_{D\in\mathcal{A}(N)}\Re^{2}\left(\frac{L'}{L}(1+it,\pi\otimes\chi_{D})\right)\ll 1,\] which follows from Lemma \ref{2ndmoment}. Therefore, 
\begin{align*}
\int_{-r}^r\left|\frac{\Phi_{\mathcal{A}(N)}(u,0)-{\Phi_{\mathrm{rand}}}(u,0)}{u}\right|\;du\ll r=(\log N)^{-B},
\end{align*}
and so \begin{equation}\label{eqn:I2}I_2\ll \max((\log N)^{-B},(\log N)^{-A}\log\log N).\end{equation} 

Similarly, we have  \begin{equation}\label{eqn:I3}I_3\ll \max((\log N)^{-B},(\log N)^{-A}\log\log N).\end{equation}

Taking $A=B=2$ and applying the estimates \eqref{eqn:I1}, \eqref{eqn:I2}, and \eqref{eqn:I3}  in  \eqref{eqn:barry-esseen} give
\[\sup_{(x,y)\in\mathbb{R}^2}|F_{\mathcal{A}(N)}(x,y)-G_{\mathrm{rand}}(x,y)|\ll \frac{(\log\log N)^2}{\log N}.\]

Next let $\mathcal{R}=[a,b]\times[c,d]$, then we can write
\begin{align*}
|\mu_{\mathcal{A}(N)}(\mathcal{R})-\mu_{\mathrm{rand}}(\mathcal{R})|&\leq |F_{\mathcal{A}(N)}(b,d)-G_{\mathrm{rand}}(b,d)| -|F_{\mathcal{A}(N)}(a,d)-G_{\mathrm{rand}}(a,d)| \\&\hspace{2em}-|F_{\mathcal{A}(N)}(b,c)-G_{\mathrm{rand}}(b,c)| +|F_{\mathcal{A}(N)}(a,c)-G_{\mathrm{rand}}(a,c)|.
\end{align*}
It follows that 
\begin{equation}\label{eqn:disc-excep}
\begin{split}
\sup_{\mathcal{R}\subset \mathbb{C}}|\mu_{\mathcal{A}(N)}(\mathcal{R})-\mu_{\mathrm{rand}}(\mathcal{R})|&\ll \frac{(\log\log N)^2}{\log N},
\end{split}
\end{equation}
where $\mathcal{R}$ varies over all rectangles in $\mathbb{C}$ with sides parallel to the axes.

Finally since $\mathcal{A}(N)=\mathcal{F}(N) \setminus \mathcal{E}(N)$ and, by Theorem \ref{zero-density}, $\mathcal{E}(N) \ll N^{\frac34}$, we have 
\begin{equation}\label{eqn:main-thm-1}
\begin{split}
\sup_{\mathcal{R}\subset \mathbb{C}}|\mu_{\mathcal{F}(N)}(\mathcal{R})-\mu_{\mathrm{rand}}(\mathcal{R})|&
=\sup_{\mathcal{R}\subset \mathbb{C}}\left|\frac{1}{|\mathcal{F}(N)|}\sum_{D\in\mathcal{E}(N)}\mathbf{1}_{\mathcal{R}}\left(-\frac{L'}{L}(1+it,\pi\otimes\chi_{D})\right)\right|\\&+\sup_{\mathcal{R}\subset \mathbb{C}}\left|\frac{|\mathcal{A}(N)|/|\mathcal{F}(N)|}{|\mathcal{A}(N)|}\sum_{D\in\mathcal{A}(N)}\mathbf{1}_{\mathcal{R}}\left(-\frac{L'}{L}(1+it,\pi\otimes\chi_{D})\right)-\mu_{\mathrm{rand}}(\mathcal{R})\right|
\\&\ll N^{-\frac14}+\sup_{\mathcal{R}\subset \mathbb{C}}\left|\frac{1}{|\mathcal{A}(N)|}\sum_{D\in\mathcal{A}(N)}\mathbf{1}_{\mathcal{R}}\left(-\frac{L'}{L}(1+it,\pi\otimes\chi_{D})\right)-\mu_{\mathrm{rand}}(\mathcal{R})\right|.
\end{split}
\end{equation}
Now from \eqref{eqn:disc-excep} and \eqref{eqn:main-thm-1}, we conclude that
\begin{align*}\sup_{\mathcal{R}\subset \mathbb{C}}|\mu_{\mathcal{F}(N)}(\mathcal{R})-\mu_{\mathrm{rand}}(\mathcal{R})|\ll \frac{(\log\log N)^{2}}{\log N}.
\end{align*}
\end{proof}


\section{Proof of Theorem \ref{min-values}}\label{sec:min-values}

In this section, we prove Theorem \ref{min-values}. We need the following lemmas, the second of which (Lemma \ref{lem-positivity}) asserts that the smooth density function $M_{\mathrm{rand}}$ associated with $\mu_{\mathrm{rand}}$ is positive. For the proof of Lemma \ref{lem-positivity}, we follow in parts the idea sketched in the remark on \cite[page~60]{JW}. We use the notion of convolution and infinite convolution of distribution functions, the definitions of which can be found in \cite[Sections~2,~4]{JW} or \cite[Section 2]{AH} among other references.

\begin{lem}\label{series-density}
Let $\sum x_n$ be a conditionally convergent series of real numbers. Let $\alpha$ and $\epsilon>0$ be fixed. Then there exists a series $\sum y_n$, where $y_n\in\{0,x_n,-x_n\}$ such that \[|\sum y_n-\alpha|<\epsilon.\]
\end{lem}

\begin{proof} Let $\alpha>0$ and $0<\epsilon<2\alpha$ be given. By our assumptions, there exist positive integers $k$ and $N$ large enough such that \[\alpha-\frac{\epsilon}{2}\leq \sum_{n=N}^{k} |x_n|\leq \alpha+\frac{\epsilon}{2}\] and 
\[-\frac{\epsilon}{2}\leq\sum_{n=k+1}^{\infty}x_n\leq \frac{\epsilon}{2}.\] We get the desired result by choosing \[y_n=\begin{cases}0&\text{if}\; 1\leq n\leq N-1,\\
|x_n|&\text{if}\; N\leq n\leq k,\\
x_n & \text{if}\; n\geq k+1.\end{cases}\]
If $\alpha<0$, we get the result by applying the previous argument to $-\alpha$. Finally, if $\alpha=0$ we set \[y_n=\begin{cases}0&\text{if}\; 1\leq n\leq N-1,\\
x_n & \text{if}\; n\geq N,\end{cases}\] where $N$ is chosen so that $-\epsilon\leq\sum_{n=N}^{\infty}x_n\leq \epsilon.$
\end{proof}

\begin{lem}\label{lem-positivity} Let $\mathbb{X}_p$ be a sequence of discrete random variables such that 
\[\mathbb{P}(\mathbb{X}_p=a)=\begin{cases}a_p&\text{if}\; a=\pm1,\\b_p&\text{if}\; a=0,\end{cases}\]
where $0<a_p,b_p<1$ and $2a_p+b_p=1$. Let $f(z)=\sum_{k=1}^{\infty}l_kz^k$ be an analytic function in a disk $|z|<\rho$. Consider the random series \[\sum_{p}\lambda_pf(r_p\mathbb{X}_p),\] where the sequences $\{\lambda_p\}$ and $\{r_p\}$ are such that $0<|r_p|<\rho$ and $\lambda_pf(r_p\mathbb{X}_p)\in\mathbb{R}$. Assume that this random series is almost surely convergent with distribution function $F$. 
In addition, assume that \[F=*_pF_p=\left(*_{p\in P_1}F_p\right)\left(*_{p\in P_2}F_p\right)\] is the infinite convolution of the distribution functions attached to the random variables $\lambda_pf(r_p\mathbb{X}_p)$, and $P_1$ and $P_2$ are two disjoint subsets of primes such that $P_1\cup P_2$ is the set of all prime numbers. Assume that $F$, $*_{p\in P_1}F_p$, and $*_{p\in P_2}F_p$ are absolutely continuous  with continuous density functions $h$, $h_1$ and $h_2$ respectively. If $\sum_{p}\sum_{k=1}^{\infty}|l_k\lambda_pr_p^k|$ is divergent, then $h(x)>0$ for all $x\in\mathbb{R}$.\end{lem}
\begin{remark}
Note that a sufficient condition for the divergence of $\sum_{p,k\geq1}|l_k\lambda_pr_p^k|$ is that $l_1\neq0$, $\sum_{p}|
\lambda_pr_p|=\infty$, and $\sum_{p,k\geq2}|l_k\lambda_pr_p^k|<\infty$.
\end{remark}
\begin{proof}[Proof of Lemma \ref{lem-positivity}]
Since $\sum_{p,k\geq1}|l_k\lambda_pr_p^k|$ is divergent and $\sum_p\lambda_pf(r_p\mathbb{X}_p)$ is almost surely convergent, then \[\left\{\sum_p\lambda_pf(r_p\mathbb{X}_p)<\infty,\; \mathbb{X}_p\in\{-1,0,1\}\right\}\] is dense in $\mathbb{R}$ by Lemma \ref{series-density}. Thus, by \cite[Proposition~B.10.8]{kowalski}, the support of $\sum_p\lambda_pf(r_p\mathbb{X}_p)$ is all of $\mathbb{R}$. 
Therefore, $\int_{a}^{b}h(x)\;dx>0$ for all $a,b\in\mathbb{R}$ with $a<b$. Hence, $h(u)$ is not identically zero on any interval $(a,b)$. Now since $\sum_{p}\sum_{k=1}^{\infty}|l_k\lambda_pr_p^k|$ is divergent, we can assume without loss of generality that $\sum_{p\in P_1}\sum_{k=1}^{\infty}|l_k\lambda_pr_p^k|$ is divergent as well. By another application of Lemma \ref{series-density} and \cite[Proposition~B.10.8]{kowalski}, we have that $h_1(u)$ is not identically zero on any interval $(a,b)$. By our assumption, we have \[h(x)=\int_{\mathbb{R}}h_{1}(x-u)h_2(u)\; du.\]
Since $\int_{\mathbb{R}} h_2(u)\;du=1$ and $h_2$ is continuous, there is an interval $(c,d)$ such that $h_2(u)>0$ on $(c,d)$. Since also $h_1(x-u)>0$ on some subinterval $(c',d')$ of $(c,d)$, then
\[h(x)=\int_{\mathbb{R}}h_1(x-u)h_2(u)\;du\geq \int_{c'}^{d'}h_1(x-u)h_2(u)\;du>0.\]
Thus, $h(x)>0$ for any $x\in\mathbb{R}$.
\end{proof}
Applying this lemma to the random series \[-\mathrm{Ld}(1+it,\pi,\mathbb{X})=\sum_{p}\log p\sum_{j=1}^d\frac{\frac{\alpha_{j,\pi}(p)}{p^{1+it}}\mathbb{X}_{p}}{1-\frac{\alpha_{j,\pi}(p)}{p^{1+it}}\mathbb{X}_{p}}\]  yields:
\begin{cor}\label{positivity} Let $t\in\mathbb{R}$ be fixed, and let $\pi$ be a cuspidal automorphic representation of $\mathrm{GL}_{d}(\mathbb{A}_{\mathbb{Q}})$ with unitary central character such that $\pi\cong\tilde{\pi}\otimes\left|\mathrm{det}\right|^{2it}$.  Then $-\mathrm{Ld}(1+it,\pi,\mathbb{X})$ has an absolutely continuous distribution with a positive density function.
\end{cor}

We are now ready to prove Theorem \ref{min-values}.

\begin{proof}[Proof of Theorem \ref{min-values}]

Let $\eta=\eta(N)$ be a positive parameter which will be chosen so that $\eta(N)\to0$ as $N\to\infty$. Let \[\Psi_N(\eta)=\left|\left\{D\in\mathcal{F}(N):\left|\frac{L'}{L}(1+it,\pi\otimes\chi_{D})\right|\leq \eta\right\}\right|.\] By Theorem \ref{main}, we have \[\frac{\Psi_N(\eta)}{|\mathcal{F}(N)|}= \mu_{\mathrm{rand}}\left(\left(-\eta,\eta\right)\right)+O\left( \frac{(\log \log N)^2}{\log N} \right).\]
 Let $M_{\mathrm{rand}}$ be the smooth density function associated with $\mu_{\mathrm{rand}}$. By Corollary \ref{positivity}, we know that $M_{\mathrm{rand}}(x)>0$ for all $x\in\mathbb{R}$. It follows that
\[\mu_{\mathrm{rand}}\left(\left(-\eta,\eta\right)\right)=\int_{-\eta}^{\eta}M_{\mathrm{rand}}(x)\;dx\gg \eta.\] Choosing $\eta=C \frac{(\log \log N)^2}{\log N}$ for some large enough positive constant $C$ yields 
  \[\frac{\Psi_N(\eta)}{|\mathcal{F}(N)|}\gg \frac{(\log \log N)^2}{\log N}. \] Hence, we get $m_{N}\ll \frac{(\log \log N)^2}{\log N}$ as desired.
  \end{proof}

\begin{rezabib}

\bib{AH}{article}{
   author={Akbary, A.},
   author={Hamieh, A.},
   title={Two dimensional value-distribution of cubic Hecke $L$-functions},
   journal={Proc. Amer. Math. Soc.},
   volume={149},
   date={2021},
   number={11},
   pages={4669--4684},
}


\bib{A-S}{article}{
   author={ Avdispahi\'c, M.},
   author={Smajlovi\'c, L.},
   title={On the {S}elberg orthogonality for automorphic {$L$}-functions},
   journal={Archiv der Mathematik},
   volume={94},
   date={2010},
   pages={147--154},
}



%
%

%
\bib{BJ}{article}{
   author={Bohr, H.},
   author={Jessen, B.},
   title={\"{U}ber die Werteverteilung der {R}iemannschen Zetafunktion},
   language={German},
  journal={Acta Math.},
   volume={58},
   date={1932},
   number={1},
   pages={1--55},
}

\bib{CM04}{article}{
   author={Cogdell, J.},
   author={Michel, P.},
   title={On the complex moments of symmetric power $L$-functions at $s=1$},
   journal={Int. Math. Res. Not.},
   date={2004},
   number={31},
   pages={1561--1617},
  }

%

%
%
%
%
%
%

%
%
%
%
%
%
%
%
%

\bib{G06}{article}{
   author={Golubeva, E. P.},
   title={Distribution of the values of {H}ecke $L$-functions at the point 1},
   language={Russian, with Russian summary},
   journal={Zap. Nauchn. Sem. S.-Peterburg. Otdel. Mat. Inst. Steklov.
   (POMI)},
   volume={314},
   date={2004},
   number={Anal. Teor. Chisel i Teor. Funkts. 20},
   pages={15--32, 285},
   issn={0373-2703},
   translation={
      journal={J. Math. Sci. (N.Y.)},
      volume={133},
      date={2006},
      number={6},
      pages={1611--1621},}
     
   }

\bib{GS}{article}{
   author={Granville, A.},
   author={Soundararajan, K.},
   title={The distribution of values of $L(1,\chi_d)$},
   journal={Geom. Funct. Anal.},
   volume={13},
   date={2003},
   number={5},
   pages={992--1028},
  }
  
  \bib{Hamieh-Mcclenagan}{article}{
   author={Hamieh, A.},
   author={Mcclenagan, R.},
   title={The distribution of values of $L'/L(1/2+\epsilon,\chi_D)$},
   journal={J. Number Theory},
   volume={238},
   date={2022},
   pages={1044-1062},
}

\bib{Harman-Mat}{article}{
   author={Harman, G.},
   author={Matsumoto, K.},
   title={Discrepancy estimates for the value-distribution of the Riemann
   zeta-function. IV},
   journal={J. London Math. Soc. (2)},
   volume={50},
   date={1994},
   number={1},
   pages={17--24},
}

\bib{HM}{article} {
     author={Hattori, T.},
   author={Matsumoto, K.},
   title={A limit theorem for Bohr-Jessen's probability measures of the
   Riemann zeta-function},
   journal={J. Reine Angew. Math.},
   volume={507},
   date={1999},
   pages={219--232},
}

\bib{heath-brown}{article}{
AUTHOR = { Heath-Brown, R. D.},
     TITLE = {A mean value estimate for real character sums},
   JOURNAL = {Acta. Arith.},
    VOLUME = {72},
    NUMBER = {3},
      YEAR = {1995},
     PAGES = {235--275},
     
     }

%

		
\bib{H-T}{article}{
   author={Humphries, P.},
   author={Thorner, J.},
   title={Towards a ${\rm GL}_n$ variant of the Hoheisel phenomenon},
   journal={Trans. Amer. Math. Soc.},
   volume={375},
   date={2022},
   number={3},
   pages={1801--1824},
}

\bib{H-T1}{article}{
author={Humphries, P.},
 author={Thorner, J.},
 title={Zeros in {R}ankin-{S}elberg $L$-functions in Families},
journal={preprint},
pages={54 pages},
year={2021},	
note={\url{https://arxiv.org/abs/2103.05634}},
	}

%
%
%

\bib{IK}{book}{
   author={Iwaniec, H.},
   author={Kowalski, E.},
   title={Analytic number theory},
   series={American Mathematical Society Colloquium Publications},
   volume={53},
   publisher={American Mathematical Society, Providence, RI},
   date={2004},
   pages={xii+615},
}

\bib{K-S}{article}{author={Kim, H.},
author={Shahidi, F.},
title={Functorial products for ${\rm GL}_2\times{\rm GL}_3$ and the
symmetric cube for ${\rm GL}_2$},
note={With an appendix by C. J. Bushnell and G. Henniart},
journal={Ann. of Math. (2)},
volume={155},
date={2002},
number={3},
pages={837--893},
}
	
\bib{kowalski}{misc}{
author={Kowalski, E.},
title={Lecture Notes: An Introduction to Probabilistic Number Theory},
journal={lecture Notes},
year={ Version of March 13, 2020},
note={\url{https://people.math.ethz.ch/~kowalski/probabilistic-number-theory.pdf}},
institution={ETH Zurich},}

\bib{JW}{article}{
  author={Jessen, B.},
   author={Wintner, A.},
   title={Distribution functions and the Riemann zeta function},
   journal={Trans. Amer. Math. Soc.},
   volume={38},
   date={1935},
   number={1},
   pages={48--88},
}

\bib{K-T}{article}{
author={Kaneko, I.},
author={Thorner, J.},
title={Highly uniform prime number theorems}
journal={arXiv:2203.09515}
date={2022},
pages={14 pages}
}

\bib{lamzouri2}{article}{
   author={Lamzouri, Y.},
   title={On the distribution of extreme values of zeta and $L$-functions in
   the strip $\frac12<\sigma<1$},
   journal={Int. Math. Res. Not. IMRN},
   date={2011},
   number={23},
   pages={5449--5503},
 
}

\bib{lamzouri3}{article}{
   author={Lamzouri, Y.},
   title={The distribution of Euler-Kronecker constants of quadratic fields},
   journal={J. Math. Anal. Appl.},
   volume={432},
   date={2015},
   number={2},
   pages={632--653}
}

\bib{LL21}{article}{
  author={Lamzouri, Y.},
   author={Languasco, A.},
   title={Small Values of $|L^\prime/L(1, \chi)|$},
  journal={Experimental Mathematics},
   date={2021},
   pages={15 pages},
  
}

\bib{Lf}{article}{
  author={Lamzouri, Y.},
   author={Lester, S.},
  author={Radziwi\l \l , M.},
   title={Discrepancy bounds for the distribution of the Riemann
  zeta-function and applications},
  journal={J. Anal. Math.},
   volume={139},
   date={2019},
   number={2},
   pages={453--494},
 
}

\bib{L-Y}{article}{
   author={Liu, J.},
   author={Ye, Y.},
   title={Perron's formula and the prime number theorem for automorphic
   $L$-functions},
   journal={Pure Appl. Math. Q.},
   volume={3},
   date={2007},
   number={2, Special Issue: In honor of Leon Simon.},
   pages={481--497},
 
}

\bib{loeve}{book}{
    AUTHOR = {Lo\`eve, M.},
     TITLE = {Probability theory. {I}},
    SERIES = {Graduate Texts in Mathematics, Vol. 45},
   EDITION = {Fourth},
 PUBLISHER = {Springer-Verlag, New York-Heidelberg},
      YEAR = {1977},
     PAGES = {xvii+425},
}

%
%
%


\bib{M20}{article}{
   author={Mine, M.},
   title={Probability density functions attached to random Euler products
   for automorphic $L$-functions},
   journal={Q. J. Math.},
   volume={73},
   date={2022},
   number={2},
   pages={397--442},
}

\bib{M-M}{article}{
   author={Mourtada, M.},
   author={Murty, V. Kumar},
   title={Distribution of values of $L'/L(\sigma,\chi_D)$},
   language={English, with English and Russian summaries},
   journal={Mosc. Math. J.},
   volume={15},
  date={2015},
  number={3},
   pages={497--509, 605},
}

\bib{Perelli-Pomykala}{article}{
   author={Perelli, A.},
   author={Pomyka\l a, J.},
   title={Averages of twisted elliptic $L$-functions},
   journal={Acta Arith.},
   volume={80},
   date={1997},
   number={2},
   pages={149--163},
}
\bib{R-S}{article}{
   author={Rudnick, Z.},
   author={Sarnak, P.},
   title={Zeros of principal $L$-functions and random matrix theory},
   note={A celebration of John F. Nash, Jr.},
   journal={Duke Math. J.},
   volume={81},
   date={1996},
   number={2},
   pages={269--322},
}

\bib{sadikova}{article}{
   author={Sadikova, S. M.},
   title={On two-dimensional analogs of an inequality of {E}sseen and their application
to the central limit theorem},
  JOURNAL = {Teor. Verojatnost. i Primenen.},
  FJOURNAL = {Akademija Nauk SSSR. Teorija Verojatnoste\u{\i} i ee Primenenija},
    VOLUME = {11},
   date={1966},
   pages={369--380},
}

\bib{S}{article}{
   author={Soundararajan, K.},
  title={Moments of the Riemann zeta function},
   journal={Ann. of Math. (2)},
   volume={170},
   date={2009},
  number={2},
  pages={981--993},
  
}

\bib{S-T}{article}{
   author={Soundararajan, K.},
   author={Thorner, J.},
   title={Weak subconvexity without a Ramanujan hypothesis},
   note={With an appendix by F. Brumley},
   journal={Duke Math. J.},
   volume={168},
   date={2019},
   number={7},
   pages={1231--1268},
}

\bib{Wu-Ye}{article}{
   author={Wu, J.},
   author={Ye, Y.},
   title={Hypothesis H and the prime number theorem for automorphic
   representations},
   journal={Funct. Approx. Comment. Math.},
   volume={37},
   date={2007},
   number={part 2},
   part={part 2},
   pages={461--471},
}

\bib{XZ21}{article}{
   author={Xiao, X.},
   author={Zhai, S.},
   title={Discrepancy bounds for distribution of automorphic $L$-functions},
   journal={Lith. Math. J.},
   volume={61},
   date={2021},
   number={4},
   pages={550--563},
 
}

%
%
%
%
%

\end{rezabib}

\end{document}